\documentclass[12pt]{amsart}
\usepackage{amsmath}
\usepackage[psamsfonts]{amssymb}
\usepackage[all]{xy}
\usepackage{tikz}
\usepackage{tikz-cd}
\usetikzlibrary{%
  matrix,%
  calc,%
  arrows%
}
\usepackage{graphicx}
\usepackage{dsfont}
\usepackage[T1]{fontenc}
\usepackage[bitstream-charter]{mathdesign}
\usepackage{hyperref}
\newtheorem{theorem}{Theorem}[section]
\newtheorem{prop}[theorem]{Proposition}
\newtheorem{lemma}[theorem]{Lemma}
\newtheorem{cor}[theorem]{Corollary}

\newtheorem*{type1*}{Type I}
\newtheorem*{type2*}{Type II}

\newtheorem{ex}[theorem]{Example}
\theoremstyle{remark}
\newtheorem{dfn}[theorem]{Definition}
\newtheorem{remark}[theorem]{Remark}

\def\co{\colon\thinspace}

\def\ep{\epsilon}

\def\R{\mathbb{R}}
\def\Z{\mathbb{Z}}
\def\N{\mathbb{N}}
\def\C{\mathbb{C}}

\def\K{\mathcal K}
\def\P{{\rm{Per}}}
\def\D{\rm{Diff}}
\def\S{{\rm Spec}\,}
\def\r{{\rm rank}}

\def\Ham{{\rm Ham}}
\def\CF{{\rm CF}}

\def\QH{{\rm QH}}
\def\HF{{\rm HF}}

\begin{document}
\title{Symplectic structure perturbations and continuity of symplectic invariants}

\author{Jun Zhang}
\email{junzhang@mail.tau.ac.il}
\address{School of Mathematical Sciences\\Tel Aviv University\\Ramat Aviv, Tel Aviv 69978, Israel}

\begin{abstract}
This paper studies how symplectic invariants created from Hamiltonian Floer theory change under the perturbations of symplectic structures, {\it not necessarily in the same cohomology class}. These symplectic invariants include spectral invariants, boundary depth, and (partial) symplectic quasi-states. This paper can split into two parts. In the first part, we prove some energy estimations which control the shifts of symplectic action functionals. These directly imply positive conclusions on the continuity of spectral invariants and boundary depth, in some important cases including any symplectic surface $\Sigma_{g \geq1}$ and any closed symplectic manifold $M$ with $\dim_{\K} H^2(M; \mathcal K) = 1$. This follows by applications on some rigidity of the subsets of a symplectic manifold in terms of heaviness and superheaviness, as well as on the continuity property of some symplectic capacities. In the second part, we generalize the construction in the first part to any closed symplectic manifold. In particular, to deal with the change of Novikov rings from symplectic structure perturbations, we construct a family of variant Floer chain complexes over a common Novikov-type ring. In this set-up, we define a new family of spectral invariants called $t$-spectral invariants, and prove that they are upper semicontinuous under the symplectic structure perturbations. This implies a quasi-isometric embedding from $(\R^{\infty}, |-|_{\infty})$ to $({\widetilde {\Ham}}(M, \omega), |-|_{H})$ under some dynamical assumption, imitating the main result from \cite{Ush13}.  \end{abstract}

\maketitle

\section{Introduction}

In Floer's method \cite{F89} of solving Arnold's conjecture (see also Hofer and Salamon \cite{HS95}), Floer chain complexes were constructed. Let $(M, \omega)$ be a symplectic manifold. A Floer chain complex symbolically depends on three parameters: an almost complex structure $J$ on $TM$, a Hamiltonian function $H \in C^{\infty}(\R/\Z \times M)$ and a symplectic structure $\omega$ on $M$. Conventionally, a Floer chain complex is denoted by $(\CF_*(M, J, H, \omega), \partial_{J, H, \omega})$ where $\partial_{J, H, \omega}$ is the differential. With the help of the symplectic action functional, a Floer chain complex can be viewed as a filtered chain complex $(\CF_*(M, J, H, \omega), \partial_{J, H, \omega}, \ell_{\omega, H})$ where $\ell_{\omega, H}$ is a filtration function with its values in $\R \cup \{-\infty\}$.  

A filtration function provides a ``height'' for each element in $\CF_*(M, J, H, \omega)$, and it satisfies the non-Archimedean triangle inequality. From this filtered viewpoint, some symplectic invariants were invented from a Floer chain complex or its homology, such as spectral invariants $\rho(a,H; \omega)$ (\cite{Vit92}, \cite{Sch00}, \cite{Oh05} and \cite{Ush08}), boundary depth $\beta(\phi; \omega)$ (\cite{Ush13} and \cite{UZ16}) and (partial) symplectic quasi-states $\zeta_a(H; \omega)$ (\cite{EP08} or \cite{EP09}). They have played important roles in the study of Hamiltonian dynamics as well as some rigidity properties of the subsets of a symplectic manifold. 

It is a natural question how these symplectic invariants change if the three parameters of a Floer chain complex are perturbed. It is well-known that the perturbations of almost complex structures do not affect the values of these invariants. Meanwhile, all these three invariants admit a Lipschitz continuity under the perturbations of Hamiltonian functions. The main results of this paper focus on the change of these symplectic invariants under the perturbations of symplectic structures. In this paper, our manifold $M$ is always assumed to be closed. 

Denote by $(M, \omega,  H)$ a Hamiltonian system where $\omega$ is a symplectic structure on the manifold $M$ and $H$ is a {\it non-degenerate} Hamiltonian function on $(M, \omega)$. Denote by $\Omega^2_{\tiny{\mbox{closed}}}(M)$ the set of all closed 2-forms of the manifold $M$. We call a closed 2-form $\omega'$ {\it a perturbation of $\omega$} if $\omega'$ is symplectic. We do not require that $\omega'$ and $\omega$ be in the same cohomology class. Let $\omega'$ be a perturbation of $\omega$. We call $\omega'$  {\it sufficiently close to $\omega$} if $\omega' - \omega \in \Omega^2_{\tiny{\mbox{closed}}}(M)$ is sufficiently small under a certain norm (see Section \ref{simple}). 

Before we state the main results, it is necessary to point out that some perturbations of a symplectic structure may change the Novikov  field  associated to $(M, \omega)$. Denote by $H^S_2(M)$ the image of $\pi_2(M)$ in $H_2(M; \Z)/{\rm Tor}$ under the Hurewicz map $\iota: \pi_2(M) \to H_2(M; \Z)$. Let $\mathcal K$ be a fixed field. Recall that an often used version of Novikov field $\Lambda^{\mathcal K, \Gamma_{\omega}}$ is defined by
\begin{equation} \label{dfn-nf}
\Lambda^{\mathcal K, \Gamma_{\omega}} : = \left\{ \sum_{\lambda \in  \Gamma_{\omega}} a_\lambda T^\lambda \, \bigg| a_{\lambda} \in \mathcal K \,\,\mbox{and}\,\, (\forall C \in \R) \,(\#\{a_\lambda \neq 0 \, | \, \lambda \leq C\} < \infty) \right\}
\end{equation} 
where $\Gamma_{\omega} =\{ {\rm Im}[\omega]: H^S_2(M) \to \R\} \leq \R$. Since $\Lambda^{\mathcal K, \Gamma_{\omega}}$ is the coefficient field of the Floer chain complex $(\CF_*(M, J, H, \omega), \partial_{J, H, \omega})$, a negative outcome is that, under a perturbed symplectic structure, the Floer chain complex may be defined over a different coefficient field. This makes the comparison between the symplectic invariants from this type of perturbations essentially harder than those from the perturbation of almost complex structures and of Hamiltonian functions. For instance, comparing $\rho(a, H; \omega)$ with $\rho(a, H; \omega')$ is in general ambiguous because where the class $a$ is taken is in question. In this paper, we will provide an approach to overcome this difficulty (see Section \ref{vari-fcc}). Finally, note that if our symplectic manifold $(M, \omega)$ is aspherical, that is $\pi_2(M) = 0$, then $\Lambda^{\mathcal K, \Gamma_{\omega}} = \K$ which is independent of the symplectic structures. Therefore, depending on whether $\Lambda^{\mathcal K, \Gamma_{\omega}}$ changes or not, we will state our main results in these two different situations.

\subsection{Continuity on aspherical manifolds} Here is the first main result in our paper. Recall that $\QH_*(M, \omega)$ denotes the quantum homology of $(M, \omega)$, that is $\QH_*(M, \omega) = H_*(M; \K) \otimes_{\mathcal K} \Lambda^{\K, \Gamma_{\omega}}$. In particular, if $M$ is aspherical,  then $\QH_*(M, \omega) = H_*(M; \K)$. 

\begin{theorem} \label{continuity result 1} Let $(M, \omega, H)$ be a Hamiltonian system where $M$ is aspherical. If $\omega'$ is a perturbation of $\omega$ and is sufficiently close to $\omega$, then there exists a constant $C$ such that 
\begin{itemize}
\item[(i)] $|\rho(a, H; \omega) - \rho(a, H; \omega')| \leq C |\omega - \omega'|$ for any $a \in H_*(M; \K)$;
\item[(ii)] $|\beta(H; \omega) - \beta(H; \omega')| \leq C|\omega - \omega'|$.
\end{itemize}
\end{theorem}

A standard example that satisfies the assumption in Theorem \ref{continuity result 1} is a symplectic surface $(\Sigma_g, \omega)$ with genus $g \geq 1$. Next, from a ``rescaling'' argument, a direct corollary of Theorem \ref{continuity result 1} is the following result under a slightly relaxed hypothesis on $M$. 

\begin{cor} \label{continuity result 2} Let $(M, \omega, H)$ be a Hamiltonian system where $M$ satisfies the condition $\dim_{\mathcal K} H^2(M; \mathcal K) = 1$. Denote by $[M]$ the fundamental class of $M$. If $\omega'$ is a perturbation of $\omega$ and is sufficiently close to $\omega$, then there exist constants $\tilde{C}_1$ and $\tilde{C}_2$ such that 
\begin{itemize}
\item[(i)] $|\rho([M], H; \omega) - \rho([M], H; \omega')| \leq \tilde C_1 |\omega - \omega'|$;
\item[(ii)] $|\beta(H; \omega) - \beta(H; \omega')| \leq \tilde C_2 |\omega - \omega'|$.
\end{itemize}
\end{cor}
Note that Corollary \ref{continuity result 2} covers some important cases, for instance, $\C P^n$ for any $n \in \N$. Roughly speaking, the proof of Theorem \ref{continuity result 1} splits into two steps. First, as elaborated in Section \ref{simple}, if $\omega'$ is a perturbation of $\omega$, then there exists a diffeomorphism on $M$ which ``reduces'' $\omega'$ to $\phi^* \omega'$ in the sense that $\phi^*\omega'$ coincides with $\omega$ near any Hamiltonian 1-periodic orbit of the Hamiltonian system $(M, H, \omega)$ (see Definition \ref{dfn-rp}). This considerably simplifies our discussion. Second, the desired estimation in Theorem \ref{continuity result 1} comes from a Floer type argument which is based on some new energy estimations established in Section \ref{sec-e-e}. In this second step, $\omega'$ is required to be sufficiently close to $\omega$. 

As applications of Theorem \ref{continuity result 1}, we can reprove results on some rigidity properties of the subsets of a symplectic surface $\Sigma_g$ with genus $g  \geq 1$.  To this end, we need the following definition from Definition 2.9 in \cite{Kaw14}.

\begin{dfn} \label{bsc} For a fixed symplectic manifold $(M, \omega)$ and an element $a \in \QH_*(M, \omega)$, a subset $U \subset M$ satisfies the {\it bounded spectrum condition with respect to $a$} if there exists a constant $K>0$ such that $\rho(a,H; \omega) \leq K$ for any Hamiltonian function $H$ supported in $\R/\Z \times U$. \end{dfn}

We can prove the following result.

\begin{theorem}\label{bscP} Let $(\Sigma_{g}, \omega)$ be a symplectic surface with genus $g \geq 1$. Then the disjoint union of simply connected open subsets satisfies the bounded spectrum condition for any $a \in H_*(\Sigma_{g}; \K)$. \end{theorem}

Recall that the concepts of heaviness and superheaviness of the subsets of a symplectic manifold are introduced in \cite{EP09} in order to study in a systematical way the non-displaceability properties (Definition \ref{hsh2}). The following result is a corollary of Theorem \ref{bscP}. 

\begin{cor} \label{bscPP} Let $(\Sigma_{g}, \omega)$ be a symplectic surface with genus $g \geq 1$.
\begin{itemize}
\item[(a)] For a closed subset $X \subset \Sigma_{g}$, if $\Sigma_{g} \backslash X$ is a disjoint union of simply connected open subsets, then it is $a$-superheavy for any $a \in H_*(\Sigma_{g}, \omega)$.
\item[(b)] If a closed subset $X \subset \Sigma_{g}$ is contained in a disk, then $X$ is \emph{not} $a$-heavy for any $a \in H_*(\Sigma_{g}; \K)$.
\end{itemize} 
\end{cor}
Notice that (a) includes the embedding of a wedge of circles $\bigvee_{i=1}^{2g} S^1 \hookrightarrow \Sigma_{g \geq 1}$. This is studied in Example 4.8 in Ishikawa \cite{Ish15} which generalizes the main result from \cite{Kaw14}. Also, (b) provides a topological obstruction for a closed subset to be $a$-heavy. When $a= [M]$, a necessary and sufficient condition for a subset to be $a$-heavy or $a$-superheavy on a symplectic surface is given in \cite{HRS14}.  \\

Another application of Theorem \ref{continuity result 1} and Corollary \ref{continuity result 2} is on the continuity of some symplectic capacities under the perturbations of symplectic structures. Let $A$ be a subset of $M$. Here we mainly focus on the displacement energy $e^{\omega}(A)$, Hofer-Zehnder capacity $c^{\omega}_{HZ}(A)$ and the spectral capacity $c^{\omega}_{\rho}(A)$. Recall that the displacement energy is defined by $e^{\omega}(A) : = \inf\{||H||_H\,|\, \phi_H^1(A) \cap A = \emptyset\}$ and $||-||_{H}$ is the Hofer norm. The capacities $c^{\omega}_{HZ}(A)$ and $c^{\omega}_{\rho}(A)$ are defined in Definition \ref{dfn-hz} and \ref{dfn-sc}, respectively. Here we put the symplectic structure $\omega$ in all the notations to emphasize their dependence on the symplectic structures. It is well-known as the {\it energy-capacity inequality} (see Theorem 1 in \cite{FGS05}) that 
\begin{equation} \label{e-c-i}
c^{\omega}_{HZ}(A) \leq c^{\omega}_{\rho}(A) \leq e^{\omega}(A).
 \end{equation}
 
 In general, we should not expect any continuity of $e^{\omega}(A)$ under the perturbations of symplectic structures, due to the following easy example.
 
\begin{ex} \label{ex2}
Take $(M = S^2, \omega_0 = \omega_{std})$ and let $A$ be the {\bf open} upper hemisphere. Then the subset $A$ is displaceable under $\omega_0$ and $e^{\omega_0}(A)= 2 \pi$. Consider a sequence of symplectic structures $\omega_n \to \omega_0$ with $\omega_n = \omega_0 + \alpha_n$ where $\alpha_n$ is {\it positively} supported over a nonempty open subset of $A$ and {\it negatively} supported over a nonempty open subset of $M \backslash \bar{A}$ such that the total area of $M$ remains the same for each $\omega_n$. Then $A$ is not displaceable under $\omega_n$ by an area consideration. By definition, $e^{\omega_n}(A) = \infty$, which implies that $e^{\omega}(A)$ is {\it not} upper semicontinuous at $\omega_0$. 
\end{ex}

However, we have the following ``boundedness'' conclusion. 

\begin{theorem}\label{bd-cap} Let $(M, \omega)$ be a symplectic manifold and $A \subset M$ be a subset. If $M$ is aspherical or $\dim_{\K}H^2(M; \K)=1$, then for any $\ep>0$, there exists a neighborhood $U_{\omega}(\ep)$ of $\omega$ in $\Omega^2_{\tiny{\mbox{\rm closed}}}(M)$ such that $e^{\omega'}(A) \geq c_{\rho}^{\omega}(A) - \ep \geq c_{HZ}^{\omega}(A) - \ep$ for any $\omega' \in U_{\omega}(\ep)$. \end{theorem}

\subsection{Semicontinuity on general manifolds} \label{ss-sgm} For general symplectic manifolds, we will deal with the coefficient-change problem of Floer chain complexes in the following way. First of all, due to Moser's trick, it is easy to see how the symplectic invariants change when the perturbation is formed from an {\it exact} 2-form. Then the discussion on perturbations can be simplified to be in $H^2(M; \K)$. Since $H^2(M; \K)$ is a finite-dimensional vector space, for a fixed symplectic structure $\omega$, there exist $[\omega_1], ... ,[\omega_m]$ as vertices forming a polygon $\Delta(\omega)$ in $H^2(M; \K)$ such that  each $\omega_i$ is a perturbation of $\omega$ and $[\omega] \in \Delta^{\circ}(\omega)$, the interior of $\Delta(\omega)$. Then if a perturbation $\omega'$ is sufficiently close to $\omega$, then 
\begin{equation} \label{cov-lin}
[\omega'] = t_0 [\omega] + t_1 [\omega_1] + ... t_m [\omega_m]
\end{equation}
for some nonnegative $t_0, ..., t_m$ where $\sum_{i=0}^m t_i = 1$. To simplify our discussion further, due to Proposition \ref{re-sim-set-up}, all the perturbations of $\omega$ here are assumed to coincide with $\omega$ near any Hamiltonian 1-periodic orbit of the Hamiltonian system $(M, \omega, H)$ (cf. Definition \ref{dfn-rp}). For any Hamiltonian system $(M, \omega, H)$, we give the following definition.
\begin{dfn} \label{dfn-nov-mfc} Define a {\it Novikov ring with multi-finiteness condition} as 
{\small \begin{equation*} 
\Lambda_{\Delta(\omega)} = \left\{ \sum_{A \in H^S_2(M)} a_A T^A \, \bigg| \, a_A \in \mathcal K, (\forall C \in \R)\,(\forall \omega' \in \Delta(\omega))\, (\#\{a_A \neq 0 \, | \, [\omega'](A) \leq C\} < \infty) \right\}.
\end{equation*}} 
\end{dfn}
Proposition \ref{well-dfn of bo} implies that there exists a well-defined family of (filtered) Floer chain complexes parametrized by $\Delta(\omega)$ over this common coefficient ring $\Lambda_{\Delta(\omega)}$. This family of filtered complexes is denoted by
\begin{equation} \label{Delta-Floer}
\{(\CF_{\Delta(\omega)}(M, H), \partial_{\omega'}, \ell_{\omega'})\}_{\omega' \in \Delta(\omega)}. 
\end{equation}
The construction of (\ref{Delta-Floer}) takes its inspiration from \cite{Ono05}. Moreover, for every two parameters $\omega_1, \omega_2 \in \Delta(\omega)$, Proposition \ref{st-htp} says that their corresponding filtered Floer chain complexes from (\ref{Delta-Floer}) are homotopy equivalent to each other. Therefore, there exists a well-defined homology $\HF_{\Delta(\omega)}(M, H)$ over $\Lambda_{\Delta(\omega)}$. In other words, we manage to assemble the information of Floer chain complexes or their homologies under various perturbations of $\omega$ into a family of filtered chain complexes or a single homology, but over a complicated coefficient ring. It can be shown that $\HF_{\Delta(\omega)}(M, H) \simeq H_*(M; \K) \otimes_{\K} \Lambda_{\Delta{\omega}}: = \QH_{\Delta(\omega)}(M)$. If we choose a class $a \in \QH_{\Delta(\omega)}(M)$, then there will be no ambiguity when we compare spectral invariants under different symplectic structures. In fact, from $\HF_{\Delta(\omega)}(M, H)$, we can define a $\Delta(\omega)$-family of spectral invariants denoted by $\rho_{\omega'}(a, H)$ for any $a \in \QH_{\Delta(\omega)}(M)$ (see  Definition \ref{t-si}). Our next main result is on the continuity of this new spectral invariants. 

\begin{theorem} \label{upper-semi} Let $(M, \omega, H)$ be a Hamiltonian system, and $a \in \QH_{\Delta(\omega)}(M)$ be a fixed class. Then the map from $\Delta(\omega)$ to $\R$ by $\omega' \to \rho_{\omega'}(a, H)$ is upper semicontinuous at $\omega$.\end{theorem}

Meanwhile, Lemma \ref{prop-si} compares the standard spectral invariants $\rho(a, H; \omega)$ with $\rho_{\omega}(a, H)$, and they turn out to be the same. Therefore, we directly get the following corollary on the continuity of the standard spectral invariants. 

\begin{cor} \label{cor-upper} Let $(M, \omega, H)$ be a Hamiltonian system, and $a \in \QH_{\Delta(\omega)}(M)$ be a fixed class. Then the map from $\Delta(\omega)$ to $\R$ by $\omega' \to \rho(a, H; \omega')$ is upper semicontinuous at $\omega$.\end{cor}

\begin{remark} There is an obvious question on the lower semicontinuity of the spectral invariants under the perturbations of symplectic structures. Unfortunately, the method we provide here can not conclude any positive or negative conclusion on the lower semicontinuity (see Remark \ref{lower-semi}). \end{remark}

\begin{remark} Recall that boundary depth is defined from the standard Floer chain complex. In the general set-up, it is difficult to conclude any quantitative conclusion between two standard Floer chain complexes that are defined from different symplectic structures, so we can appeal to the variant Floer chain complex (\ref{Delta-Floer}) when comparing two boundary depths. For instance, one can define a $\Delta(\omega)$-parameterized boundary depth from (\ref{Delta-Floer}). However, our method and estimations are unable to give any positive or negative conclusion on its continuity. \end{remark}

Next, we want to say a few words on  (partial) symplectic quasi-states. The method we used in this paper can't apply to study the continuity of (partial) symplectic quasi-states. The main reason is that in our discussion the Hamiltonian function $H$ should be fixed whenever a symplectic structure is perturbed. In general, one should not expect any continuity result for (partial) symplectic quasi-states. We illustrate this by the following easy example, similarly to Example \ref{ex2}. 

\begin{ex} \label{ex1} Take $(M = S^2, \omega_0 = \omega_{std})$ centered at the origin in $xyz$-coordinates. The standard equation $L = \{z=0\}$ is a heavy subset and then it is not displaceable by any Hamiltonian diffeomorphism. For any given $\ep>0$, let $S_{\ep} := \{(x, y, z) \in S^2 \,| \, z \geq - \ep\}$ and $H$ be a time-independent Hamiltonian function supported on $S_{\ep}$ such that $H|_L = 1$ and $\omega_0 ({\rm supp}(H)) = \frac{1}{2} \omega_0(S^2)$. Consider a closed 2-form $\alpha$ only {\bf positively} supported on $S^2 \backslash S_{\ep}$. For any class $a \in \QH_*(S^2, \omega_0)$ and any $\delta>0$, Definition \ref{hsh2} of a heavy subset says that
\[ \zeta_a (H; \omega_0) \geq \inf_{L} H = 1.\]
However, $\zeta_{a}(H; \omega_{0} + \delta\alpha) = 0$ because $(\omega_0+ \delta\alpha) ({\rm supp}(H)) < \frac{1}{2} (\omega_0+\delta\alpha)(S^2)$ implies that the support of $H$ is displaceable in $(S^2, \omega_0 + \delta \alpha)$. 
\end{ex}

\begin{remark} Our choice of the perturbation in Example \ref{ex1} is quite special. It is easy to see that there are plenty of other perturbations that we can take so that symplectic quasi-states are invariant (in particular change continuously). In general, it would be interesting to systematically study in which way we can perturb the symplectic structure so that symplectic quasi-states can satisfy a continuity result. \end{remark}

Recall that $d_H$ denotes Hofer's metric on the universal cover of Hamiltonian diffeomorphism group  of $(M, \omega)$ denoted by $\widetilde{\Ham}(M, \omega)$. As a standard application of Theorem \ref{upper-semi}, we have the following Theorem \ref{embedding} which is similar to Theorem 1.1 in \cite{Ush13}. It concludes a large-scale geometric property of the metric space $(\widetilde{\Ham}(M, \omega), d_H)$ when $M$ satisfies a certain dynamical condition. This dynamical condition imitates the assumption in Theorem 1.1 in \cite{Ush13}. 

\begin{theorem} \label{embedding} Suppose that a manifold $M$ admits a symplectic structure $\omega$ satisfying the saa-condition ((2) in Definition \ref{ap}). Then there exists an embedding $\Phi: \R^{\infty} \to \widetilde{\Ham}(M, \omega)$ such that for every $\vec{v}, \vec{w} \in \R^{\infty}$, we have
\[ |\vec{v}- \vec{w}|_{\infty} \leq d_{H} (\Phi(\vec{v}), \Phi(\vec{w})) \leq osc(\vec{v} - \vec{w})\]
where $|\vec{a}|_{\infty} = \max_{i} |a_i|$ and $osc(\vec{a}) = \max_{i,j}|a_i - a_j|$ for the vector $\vec{a} = (a_1, a_2, ...)$.
\end{theorem}

\subsection{Outline of the paper} 

In Section \ref{pre}, the background of Hamiltonian Floer theory and several symplectic invariants derived from this theory are briefly reviewed. In Section \ref{simple}, we explain how a general perturbation of a symplectic structure can be simplified into a reduced perturbation (see Definition \ref{dfn-rp}), and this key step is given by Proposition \ref{re-sim-set-up}. Imitating the standard homotopy argument between two Hamiltonian functions in Hamiltonian Floer theory, Section \ref{sec-e-e} discusses various energy estimations in terms of the perturbations of symplectic structures. The main results in this section are Proposition \ref{ene-est} and Proposition \ref{ene-est-cor3}. In Section \ref{sec-cont} and Section \ref{sec-app}, the main results of this paper, Theorem \ref{continuity result 1} and \ref{continuity result 2} as well as Theorem \ref{bscP} and Theorem \ref{bd-cap} on applications, are proved. In particular, the proof of Theorem \ref{bscP} takes inspiration from Ostrover's trick (\cite{Ost03}). Section \ref{vari-fcc} serves as an algebraic preparation for the discussion of any closed symplectic manifold (in the sense that Novikov fields have to be considered). In this section, we define a parametrized Floer chain complex over an extended version of the Novikov ring (see Proposition \ref{well-dfn of bo}). Based on this discussion, another main result of this paper, Theorem \ref{upper-semi}, is proved in Section \ref{sec-si}. Finally, Theorem \ref{embedding} on an application of the Hofer geometry on $\widetilde{\rm Ham}(M, \omega)$ is proved in Section \ref{sec-emb}. 

\subsection{Acknowledgement} First of all, I would like to thank my Ph.D. advisor, Michael Usher, for introducing this problem to me, as well as many helpful discussions and suggestions along the process of writing up this paper. Second, I would like to thank several useful conversations with Michael Entov, Yong-Geun Oh, Kaoru Ono, Leonid Polterovich, Sobhan Seyfaddini, and Weiwei Wu. I am also grateful for Stefan M{\"u}ller for his common interest in this problem and his generous sharing of many ideas.  Last but not least, I am thankful to the anonymous referee for his/her suggestions and corrections.

\section{Preliminaries} \label{pre}

\subsection{Floer chain complex} \label{ss-fcc} In this subsection, we will briefly review the constriction of a Floer chain complex.  In an abstract language, a Floer chain complex is an example of a filtered complex defined as follows. Denote by $\Lambda^{\K, \Gamma}$ an abstract Novikov field (over ground field $\K$), that is, 
\[ \Lambda^{\K, \Gamma} := \left\{ \sum_{\lambda \in \Gamma} a_{\lambda} T^{\lambda}\, \bigg|\, \Gamma \leq \R, a_{\lambda} \in \mathcal K \,\,\mbox{and}\,\, (\forall C \in \R) \,(\#\{a_\lambda \neq 0 \, | \, \lambda \leq C\} < \infty) \right\}.\]
By this finiteness condition on $\Lambda^{\K, \Gamma}$, there exists a well-defined valuation $\nu: \Lambda^{\K, \Gamma} \to \R \cup \{\infty\}$ which simply takes the minimal exponent $\lambda$ of any element in $\Lambda^{\mathcal K, \Gamma}$. 

\begin{dfn} \label{dfn-filter-cpx} We call $(C_*, \partial_C, \ell_C)$ a {\it filtered complex over $\Lambda^{\K, \Gamma}$} if $(C_*, \partial_C)$ is a chain complex and $C_k$ is a finite dimensional vector space over $\Lambda^{\K, \Gamma}$ for each degree $k \in \Z$. Moreover, $\ell_C: C_* \to \R \cup \{-\infty\}$ is a function that satisfies (i) the non-Archimedean triangle inequality; (ii) $\ell_C( \partial_C x) < \ell_C(x)$ and (iii) $\ell_C(\lambda x) = \ell_C(x) - \nu(\lambda)$ for any $x \in C$ and $\lambda \in \Lambda^{\K, \Gamma}$. \end{dfn}

On the level of chain complexes, we can compare two filtered complexes over the same coefficient field $\Lambda^{\K, \Gamma}$. The following definition is taken from Definition 1.3 in \cite{UZ16}. 

\begin{dfn} \label{dfn-qe} Let $(C_*,\partial_C,\ell_C)$ and $(D_*,\partial_D,\ell_D)$ be filtered complexes over $\Lambda^{\K, \Gamma}$, and $\delta\geq 0$.  A $\delta$-{\it quasiequivalence} between $C_*$ and $D_*$ is a quadruple $(\Phi,\Psi,K_1,K_2)$ where:
\begin{itemize} 
\item[(1)] $\Phi\co C_*\to D_*$ and $\Psi\co D_*\to C_*$ are chain maps, with $\ell_D(\Phi c)\leq \ell_C(c)+\delta$ and $\ell_C(\Psi d)\leq \ell_D(d)+\delta$ for all $c\in C_*$ and $d\in D_*$.
\item[(2)] $K_1\co C_*\to C_{*+1}$ and $K_2\co D_*\to D_{*+1}$ obey the homotopy equations $\Psi\circ\Phi-\mathds{1}_{C_*}=\partial_CK_1+K_1\partial_C$ and $\Phi\circ\Psi-\mathds{1}_{D_*}=\partial_DK_2+K_2\partial_D$, and for all $c\in C_*$ and $d\in D_*$ we have $\ell_C(K_1c)\leq \ell_C(c)+2\delta$ and $\ell_D(K_2d)\leq \ell_D(d)+2\delta$.
\end{itemize}
Thus we call $(C_*,\partial_C,\ell_C)$ and $(D_*,\partial_D,\ell_D)$ are $\delta$-{\it quasiequivalent} if there exists a $\delta$-quasiequivalence between them. 
\end{dfn}

\begin{ex} \label{ex-filtered-complex} Recall that the Hofer norm of any function $H \in C^{\infty}(\R/\Z \times M)$ is defined by $||H||_{H} := \int_0^1 \left(\max_M H_t - \min_M H_t\right) \,dt$. Then a Floer chain complex, denoted by $(\CF_*(M, J, H, \omega), \partial_{J, H, \omega}, \ell_{H, \omega})$, is a filtered complex over a Novikov field. Moreover, for two pairs $(J_{-}, H_{-})$ and $(J_{+}, H_{+})$, the Floer chain complexes $(\CF_*(M, J_+, H_+, \omega), \partial_{J_+, H_+, \omega}, \ell_{H_+, \omega})$ and $(\CF_*(M, J_-, H_-, \omega), \partial_{J_-, H_-, \omega}, \ell_{H_-, \omega})$ are $||H_+ - H_-||_{H}$-quasiequivalent.\end{ex}

To justify Example \ref{ex-filtered-complex}, we elaborate three objects - the generators of $\CF_*(M, J, H, \omega)$, the construction of $\partial_{J, H, \omega}$ and the constructions of homotopy $K_1$ and $K_2$ from Definition \ref{dfn-qe}.  Given a smooth function $H \in C^{\infty}(\R/\Z \times M)$ on a closed symplectic manifold $(M, \omega)$, the Hamiltonian flow $\phi_H^t$ comes from the differential equation
\[ \frac{d\phi_t}{dt}  = X_H  \circ \phi_t \,\,\,\,\,\,\mbox{where} \,\,\,\,\,\,\omega(\cdot, X_H) = d(H(t, \cdot)). \]
The generators of $\CF_*(M, J, H, \omega)$ are Hamiltonian contractible loops $\gamma: \R/\Z \to M$ where $\gamma(t) = \phi_H^t (\gamma(0))$. With a non-degeneracy condition, there are only finitely many such generators. Set $\Gamma_{\omega} = \left\{ {\rm{Im}}[\omega]: H^S_2(M) \to \R \right\} \leq \R$ where $H^S_2(M)$ is the image of $\pi_2(M)$ in $H_2(M; \Z)/{\rm Tor}$ under Hurewicz map $\iota: \pi_2(M) \to H_2(M; \Z)$, then 
\[ \CF_*(M, J, H, \omega) : = {\rm span}_{\Lambda^{\K, \Gamma_{\omega}}} \left< \mbox{Hamiltonian contractible loops} \right> .\]

The grading of each generator is given by Conley-Zehnder index $\mu_{CZ, \omega}$. Its explicit definition can be found in \cite{RS93}. Moreover, since each generator $\gamma$ is a contractible loop, we can {\it fix} a disk $w$ spanning $\gamma$ and assign a value to each pair $(\gamma, w)$ via the symplectic action functional defined as, 
\begin{equation} \label{action-ori}
 \mathcal A_{H, \omega}((\gamma, w)) = -\int_{D^2} w^* \omega + \int_0^1 H(t, \gamma(t)) dt.
 \end{equation}
Then define $\ell_{H, \omega}: \CF_*(M, J, H, \omega) \to \R \cup \{-\infty\}$ by 
\begin{equation} \label{dfn-filtration}
\ell_{H, \omega} \left(\sum_i \lambda_i (\gamma_i, w_i) \right) = \max_i \{\mathcal A_{H, \omega} (\gamma_i, w_i) - \nu(\lambda_i)\}. 
\end{equation}
Conventionally, to simplify the discussion we usually replace the pair $(\gamma, w)$ with an equivalence class $[\gamma, w]$. We choose to define that $(\gamma, w)$ is equivalent to $(\tau, v)$ if and only $\gamma(t) = \tau(t)$ and $[w \# (-v)]$ is homologically trivial. Note that the symplectic action functional $\mathcal A_{H, \omega}$ and Conley-Zehnder index $\mu_{CZ, \omega}$ are both well-defined over $[\gamma, w]$. Hence $\ell_{H,\omega}$ is well-defined over any linear combination of the equivalence classes. 

\begin{remark} In most literature, the equivalence relation between $(\gamma, w)$ and $(\tau, v)$ is weaker than what is given above. Explicitly, $(\gamma, w)$ is equivalent to $(\tau, v)$ if and only if $\gamma(t) = \tau(t)$ and $[w \# (-v)] \in \ker([\omega]) \cap \ker(c_1)$ where $c_1$ is the first Chern class of $(M, \omega)$. Since later in this paper the perturbations of symplectic structures are considered, we need an equivalence relation that is independent of symplectic structures. \end{remark}

The Floer boundary operator $\partial_{J, H, \omega}: \CF_*(M, J, H, \omega) \to \CF_{*-1}(M, J, H, \omega)$ is defined by counting solutions (modulo $\R$-translation) of the partial differential equation 
\begin{equation} \label{dfn-boundarymap}
 \frac{\partial u}{\partial s} + J_t(u(s,t)) \left( \frac{\partial u}{\partial t} - X_H(t, u(s,t)) \right) =0, 
 \end{equation} 
where $\{J_t\}_{0 \leq t \leq 1}$ is a family of almost complex structures that are compatible with $\omega$ and $u(s,t): \R \times \R/\Z \to X$ is such that  \begin{itemize}
\item{} $u$ has finite energy $E(u) = \int_{\R \times \R/\Z} \left|\frac{\partial u}{\partial s} \right|^2 dtds$;
\item{} $u$ has asymptotic condition $u(s,\cdot) \to \gamma_{\pm} (\cdot)$ as $s \to \pm \infty$;
\item{} $\mu_{CZ}([\gamma_-,w_-]) -  \mu_{CZ}([\gamma_+,w_+]) =1$ and $[\gamma_+, w_+] = [\gamma_+, w_- \#u]$.
\end{itemize} 
The celebrated Gromov compactness theorem guarantees that $\partial_{J,H,\omega}$ is well-defined over $\Lambda^{\mathcal K, \Gamma_{\omega}}$. Moreover, Floer chain complex $(\CF_*(M, J, H, \omega), \partial_{J, H, \omega})$ defines a homology $\HF_*(M, J, H, \omega)$ called Floer homology which only depends on the manifold $M$ itself, so is denoted as $\HF_*(M)$. Explicitly, up to a degree shift, one gets that $\HF_*(M) \simeq H_*(M,\K) \otimes_{\mathcal K} \Lambda^{\K, \Gamma_{\omega}}$ where the right hand side is called the quantum homology of $M$ and denoted as $\QH_*(M, \omega)$. A standard way to prove this isomorphism is via PSS-map, denoted as $PSS_*$. For its explicit construction, see \cite{PSS96}. 

Finally, given two different pairs $(J_{-}, H_{-})$ and $(J_{+}, H_{+})$, consider a homotopy $(\mathcal J, \mathcal H)$ where $\mathcal H = \mathcal H_s = \mathcal H(s,t,x): \R \times \R/\Z \times M \to \R$ between $H_{-}(t,x)$ and $H_{+}(t,x)$ in the form of 
\[ \mathcal H(s, t, x) = (1- \alpha(s)) H_{-}(t,x) + \alpha(s) H_{+}(t,x), \]
where $\alpha(s)$ is a cut-off function, i.e., $\alpha(s) = 0$ for $s \in (-\infty, 0]$, $\alpha(s) = 1$ for $s \in [1, \infty)$ and $0\leq \alpha'(s) \leq 1$ for $s \in (0,1)$; $\mathcal J = \mathcal J_s$ is a homotopy (compatible with $\omega$) between $J_-$ and $J_{+}$ by a cut-off function, too. Then similarly to the boundary operator $\partial_{J, H, \omega}$, the construction of continuation map $\Phi$ is by counting solutions $u(s,t): \R \times \R/\Z \to M$ of a parametrized pseudoholomorphic equation 
\begin{equation} \label{continuation-map}
\frac{\partial u}{\partial s} + \mathcal J_s(u(s,t)) \left( \frac{\partial u}{\partial t} - X_{\mathcal H_s}(t, u(s,t)) \right)  =0 
\end{equation}
such that  \begin{itemize}
\item{} $u$ has finite energy $E(u) = \int_{\R \times \R/\Z} \left|\frac{\partial u}{\partial s} \right|^2 dtds$;
\item{} $u$ has asymptotic condition $u(s,\cdot) \to \gamma_{\pm} (\cdot)$ as $s \to \pm \infty$;
\item{} $\mu_{CZ}([\gamma_-,w_-]) -  \mu_{CZ}([\gamma_+,w_+]) =0$ and $[\gamma_+, w_+] = [\gamma_+, w_- \#u]$.
\end{itemize} 
There are two well-known facts (\cite{Sal97}). One is that $\Phi$ is a chain map by a gluing argument; the other is that, for any other homotopy $(\mathcal J', \mathcal H')$, the associated chain map $\Phi'$ is chain homotopic to $\Phi$. Here we give some details on the shift of symplectic actions. Since the symplectic structures are the same, for brevity, denote the symplectic action functional by $\mathcal A_{H}$ if the Hamiltonian function is $H$. The standard computation goes as follows.
\begin{align*} 
\mathcal A_{H_+}([\gamma_+, w_+]) - \mathcal A_{H_-}([\gamma_-, w_-]) & = \int_{-\infty}^{\infty} \frac{d}{ds} \mathcal A_{\mathcal H_s}([u(s,\cdot), (w_- \#u)(s,\cdot)]) ds \\
&= -E(u) + \int_{-\infty}^{\infty}\int_0^1 \alpha'(s) (H_+ - H_-) (t, u(s,t))dtds \\
& \leq -E(u) + \int_0^1 \max_M (H_+ - H_-) (t ,u(s, t))dt\\
& \leq  \int_0^1 \max_M (H_+ - H_-) dt.
\end{align*}
Strictly speaking, to pass from the upper bound $\int_0^1 \max_M (H_+ - H_-) dt$ to $||H_+ - H_-||_H$, we need to normalize $H_+$ and $H_-$ so that both have mean values zero over $M$ (and this can be done simply by a constant shift due to our assumption that $M$ is closed). This implies that $\int_0^1 \min_M (H_+ - H_-)dt \leq 0$ and then $\int_0^1 \max_M (H_+ - H_-) dt  \leq ||H_+ - H_-||_H$.
 
 \subsection{Some symplectic invariants}
\subsubsection{Spectral invariant}
The filtration $\ell_{H, \omega}$ defined in (\ref{dfn-filtration}) can be used to define a measurement for elements in $\HF_*(M)$, and the outcomes of this measurement are called  {\it spectral invariants}.
\begin{dfn}
For any $a \in \QH_*(M, \omega)$, define its spectral invariant with respect to the Hamiltonian system $(M, \omega, H)$ by 
\[ \rho(a, H; \omega) := \inf \{\ell_{H, \omega} (\alpha) \,| \, \alpha \in CF_*(M, J, H, \omega) \,\,\,\mbox{with} \,\,\, [\alpha] = PSS_*(a) \}. \]
\end{dfn}
Recall that $PSS_*$ is the well-known isomorphism from $\QH_*(M, \omega)$ and $\HF_*(M)$. It is easy to see that spectral invariants are independent of the almost complex structures. Also, they enjoy many useful properties. The following result will be used later, which is Theorem 1.4 in \cite{Ush08}.

\begin{theorem} \label{realization thm} Let $(M, \omega, H)$ be a Hamiltonian system. Then for any $a \in \QH_*(M, \omega)$, there exists some $\alpha \in \CF_*(M, J, H, \omega)$ such that $[\alpha] = PSS_*(a)$ and $\rho(a, H; \omega) = \ell_{H, \omega}(\alpha)$. In other words, define $\S(H, \omega) : = \{\ell_{H, \omega}(\alpha) \,| \, \alpha \in \CF_*(M, J, H, \omega)\}$, then $\rho(a, H; \omega)  \in \S(H, \omega)$. \end{theorem} 

\subsubsection{Boundary depth}
Boundary depth is defined on the chain complex level. 
\begin{dfn} For a Floer chain complex $(\CF_*(M, J,H, \omega), \partial_{J, H, \omega}, \ell_{H, \omega})$, define its boundary depth by
\[ \beta(H; \omega) := \sup_{x \in {\rm{Im}} \partial_{J, H, \omega}} \inf \left\{ \ell_{H, \omega}(y) - \ell_{H, \omega}(x) \,| \, \partial_{J, H, \omega} y = x \right\}. \]
More generally, the same formula as above defines the boundary depth $\beta$ for any filtered complex $(C, \partial_C, \ell_C)$. 
\end{dfn}
By Remark 3.3 in \cite{Ush11}, boundary depth is independent of the almost complex structures. What needs to be emphasized is that though boundary depth is even well-defined on $\Ham(M, \omega)$ by Corollary 5.4 in \cite{Ush13}, we still use the notation $\beta(H; \omega)$. The main reason is that the process that passes from a Hamiltonian function $H$ to its corresponding Hamiltonian diffeomorphism $\phi = \phi_H^1 \in \Ham(M, \omega)$ depends on the symplectic structure $\omega$. Similarly to spectral invariants, boundary depth also enjoys many useful properties. The following ones will be used later. They are Theorem 7.4 in \cite{Ush13} and Proposition 3.8 in \cite{Ush13}, respectively. 

\begin{theorem} \label{thm-bd-re} For any Floer chain complex $(\CF_*(M, J, H, \omega), \partial_{J, H, \omega}, \ell_{H, \omega})$, there exists $y \in \CF_*(M, J, H, \omega)$ such that $\beta(H; \omega) = \ell_{H, \omega}(y) - \ell_{H, \omega}(\partial y).$ In other words, $\beta(H; \omega) \in \S^{\Delta}(H, \omega): = \{ s- t\,| \, s, t \in \S(H, \omega)\}$. \end{theorem}

\begin{theorem} \label{bd-qe} Let $(C_1, \partial_1, \ell_1)$ and $(C_2, \partial_2, \ell_2)$ be two filtered complexes. Denote their boundary depths by $\beta_1$ and $\beta_2$, respectively. If these two filtered chain complexes are $\delta$-quasiequivalent, then $|\beta_1 - \beta_2|\leq \delta$. \end{theorem} 

\subsubsection{(Partial) symplectic quasi-state}
In general, any stable homogenous quasi-morphism (\cite{EP03}, \cite{EP08}) induces a quasi-state, i.e., a functional $\zeta: C^{\infty}(M) \to \R$ satisfying the following properties:
\begin{itemize}
\item If $\{F, G\}=0$, then $\zeta(H+aG) = \zeta(H) + a\zeta(G)$ for any $a \in \R$.
\item If $H \leq G$, then $\zeta(H) \leq \zeta(K)$.
\item $\zeta(1) = 1$. 
\end{itemize}
In particular, if we use {\it spectral quasi-morphism} which is constructed from spectral invariants (\cite{EP03}), we can get (partial) symplectic quasi-states. More directly, 
\begin{dfn} \label{dfn-pqs} For any $a \in \QH_*(M, \omega)$, define a function $\zeta_a(-; \omega): C^{\infty}(M) \to \R$ by 
\[ \zeta_a(H; \omega) = \lim_{k \to \infty} \frac{\rho(a, kH; \omega)}{k}.\]
This function is called the (partial) symplectic quasi-state associated to the class $a$. \end{dfn}
Symplectic quasi-states are powerful tools to study the rigidity of intersections of subsets in a symplectic manifold. The closely related concepts are \emph{heavy subset} and \emph{superheavy subset}. 

\begin{dfn} \label{hsh2} For a given $a \in \QH_*(M, \omega)$, we call a closed subset $X \subset M$ {\it $a$-heavy} if $\zeta_a(H; \omega) \geq \inf_X H$ for all $H \in C^{\infty}(M)$ and {\it $a$-superheavy} if $\zeta_a(H; \omega) \leq \sup_X H$ for all $H \in C^{\infty}(M)$. \end{dfn}

To end this section, we emphasize that the quasiequivalence conclusion for Floer chain complexes in Example \ref{ex-filtered-complex} readily implies that all three symplectic invariants introduced above satisfy 1-Lipschitz continuity under the perturbations of Hamiltonian functions. Here we summarize them into the following theorem, and they are from (5) in Theorem I in \cite{Oh05}, (iii) in Theorem 1.4 in \cite{Ush13} and Theorem 3.2 in \cite{Ent14}, respectively. 
\begin{theorem}\label{Lip}We have the following continuity results, where $||-||_{H}$ is the Hofer norm (cf. Example \ref{ex-filtered-complex}).
\begin{itemize} 
\item[(a)]  Given any $a \in \QH_*(M, \omega)$, for any $H, G \in C^{\infty}(\R/\Z \times M)$, we have 
\begin{equation*} 
 |\rho(a, H; \omega)- \rho(a, G; \omega)| \leq ||H-G||_H. 
 \end{equation*}
\item[(b)] For any $H, G \in C^{\infty}(\R/\Z \times M)$, we have 
\begin{equation*} 
 |\beta(H; \omega) - \beta(G; \omega)| \leq ||H-G||_H. 
 \end{equation*}
\item[(c)]  Given any idempotent element $a \in \QH_*(M, \omega)$, for any $H, G \in C^{\infty}(M)$, 
\begin{equation*} 
\min_{M}(H-G) \leq \zeta_a(H; \omega) - \zeta_a(G; \omega) \leq \max_{M} (H-G). 
\end{equation*}
\end{itemize}
\end{theorem}

\section{Reduced perturbation} \label{simple}

Given a Hamiltonian system $(M, \omega, H)$, denote by ${\P}(\omega, H)$ the collection of all non-constant geometrically distinct Hamiltonian 1-periodic orbits, where $x(t)$ denotes a generic element in ${\P}(\omega, H)$. Since $H$ is non-degenerate and $M$ is assumed to be closed, ${\P}(\omega, H)$ contains only finitely many elements. Consider the following subset of $\Omega^2_{\tiny{\mbox{closed}}}(M)$, 
\begin{equation} \label{dfn-space-rp} \Omega_{\omega, H} = \left\{\alpha \in \Omega^2_{\tiny{\mbox{closed}}}(M) \,\bigg|\,\begin{array}{cc} \mbox{ $\alpha$ vanishes in a neighborhood} \\ \mbox{of each $x(t) \in {\P}(\omega, H)$} \end{array}\right\}. 
\end{equation}
This leads to the following definition. 
\begin{dfn} \label{dfn-rp} Given a Hamiltonian system $(M, \omega, H)$, we call a perturbation $\omega'$ {\it a reduced perturbation of $\omega$} if $\omega' = \omega + \alpha$ for some $\alpha \in \Omega_{\omega, H}$. \end{dfn}

In this section, we will explain how any perturbation of a given symplectic structure $\omega$ can be reduced to a reduced perturbation as defined in Definition \ref{dfn-rp}. Explicitly, we have the following result.  

\begin{prop} \label{re-sim-set-up} Let $(M, \omega, H)$ be a Hamiltonian system. If $\omega'$ is a perturbation of $\omega$, then there exists a diffeomorphism $\phi \in {\D}(M)$ such that $\phi^*\omega' - \omega \in \Omega_{\omega, H}$ and $d_{C^0}(\phi, \mathds{1}_M) \leq C |\omega' - \omega|$ for some constant $C$ that does not depend on $\omega'$.\end{prop}

Here let us elaborate on the measurement $|-|$ on the closed 2-forms in Proposition \ref{re-sim-set-up}. It is well-known that for any $k \in \Z$, we can associate the {\it $k$-norm} on the space of closed differential $n$-forms $\Omega_{\tiny{\rm closed}}^n(M)$ for any $n \in \Z$. Explicitly, fixing a local chart $\{(U_i, \phi_i)\}_{i=1}^m$ of $M$, any $\alpha \in \Omega^n(M)$ can be locally expressed as 
\[ (\phi_i^{-1})^* (\alpha|_{U_i}) = \sum_{(s,t) \in \{1, ..., 2n\} \times \{1, ..., 2n\} } f_{i,s,t} dx_s \wedge dx_t\]
where $\{x_1, ..., x_{2n}\}$ is the coordinate of $\R^{2n}$ and $f_{i,s,t}: \phi_i(U_i) \to \R$. Then define the $k$-norm
\[ ||\alpha||_k : = \max_{i\in \{1, ..., m\}} \max_{(s,t) \in \{1, ..., 2n\} \times \{1, ..., 2n\} } \max_{l \leq k} ||f_{i,s,t}||_{C^l}\]
where $||-||_{C^l}$ is the standard $C^l$-norm defined over the function space. In particular, $(\Omega^n_{{\tiny{\rm closed}}}(M), ||\cdot||_k)$ is a normed vector space. Moreover, there exists a sequence of positive real numbers $\vec{\ep}= (\ep_k)_{k \geq 0}$ such that under the {\it $\vec{\ep}$-norm} defined by $||\alpha||_{\vec{\ep}} := \sum_{k \geq 0} \ep_k || \alpha||_k$, the space $\Omega^n_{{\tiny{\rm closed}}}(M)_{\vec{\ep}} = \{ \alpha \in  \Omega^n_{{\tiny{\rm closed}}}(M) \,| \, ||\alpha||_{\vec{\ep}} < \infty \}$ is a complete normed vector space under $||-||_{\vec{\ep}}$. For brevity, we denote $||-||_{\vec{\ep}}$ as $|-|$. In a similar way, we can define a semi-norm on $H_{dR}^n(M; \R)$ as follows. Given any $a \in H^n_{dR}(M; \R)$, define $|a|_h = \inf \{ |\alpha|  \,| \, [ \alpha ] = a\}$.

\begin{remark} Due to the $\vec{\ep}$-norm, our proposed space $\Omega_{\omega, H}$ defined in (\ref{dfn-space-rp}) should be modified to be $\Omega_{\omega, H} \cap \Omega^2_{{\tiny{\rm closed}}}(M)_{\vec{\ep}}$. For brevity, we still use notation $\Omega_{\omega, H}$ as well as the corresponding definition of the reduced perturbations in Definition \ref{dfn-rp}. \end{remark}
 
Before giving the proof of Proposition \ref{re-sim-set-up}, we want to explain how we use  it. In order to compare the symplectic invariants under different symplectic structures, we need to compare their associated Floer chain complexes, that is, $(\CF_*(M, J, H, \omega), \partial_{J, H, \omega})$ and $(\CF_*(M, J, H, \omega'), \partial_{J, H, \omega'})$ where $\omega'$ is a perturbation of $\omega$. We can simplify this procedure by taking the $\phi$ concluded from Proposition \ref{re-sim-set-up} and inserting two intermediate steps, 
\[ \xymatrix{
    \CF_*(M, J, H, \omega) \ar@{-}[r]^-{(1)} \ar@{--}[d]_-{(4)} & \CF_*(M, J, H, \omega') \ar@{--}[d]^-{(2)} \\
     \CF_*(M, J, H, \phi^*\omega') \ar@{--}[r]_-{(3)}  & \CF_*(M, \phi^*J, \phi^*H, \phi^* \omega')
    }\]
where (1) is the desired comparison. Observe that (2) will not change spectral invariants, boundary depth, and (partial) symplectic quasi-states. Moreover, (3) only results in a small difference due to the second conclusion of Proposition \ref{re-sim-set-up} and 1-Lipschitz continuities from Theorem \ref{Lip}. In other words, the original comparison (1) can be replaced by (4) if we forgive the small defect from (3). Hence, by the first conclusion of Proposition \ref{re-sim-set-up}, we only need to consider the reduced perturbations. Moreover, if $\omega'$ is any reduced perturbation of $\omega$, the following proposition says that the generators of $\CF_*(M, J, H, \omega)$ and $\CF_*(M, J, H, \omega')$ are the same, and also their degrees are the same. However, their boundary operators are different in general. 

\begin{prop} \label{no-other-orbits} Let $(M, \omega, H)$ be a Hamiltonian system. Then for any reduced perturbation $\omega'$ of $\omega$, ${\P}(\omega, H) = {\P}(\omega', H)$. Moreover, $\mu_{CZ, \omega}(x(t)) = \mu_{CZ, \omega'}(x(t))$ for each contractible Hamiltonian 1-periodic orbit $x(t) \in {\P}(\omega, H)$. \end{prop}

\begin{proof} By definition, since $\omega'$ coincides with $\omega$ on $\cup_i U_i$ where $U_i$ is a neighborhood of $x_i(t) \in \P(\omega, H)$, ${\P}(\omega', H)$ has at least as many Hamiltonian 1-periodic orbits as ${\P}(\omega, H)$ has. Next, we claim that there are no other orbits in ${\P}(\omega', H)$ outside $\cup_i U_i$. Without loss of generality, we assume ${\P}(\omega,H)$ consists of only one element $x(t)$ with a neighborhood $U$. We will prove our claim by contrapositive. Suppose there exists a sequence of symplectic structures $\{\omega_\frac{1}{n}\}$ approaching to $\omega$ and for each $\omega_\frac{1}{n}$, there exists some $z_n(t)$ such that $z_n(t) \in {\P}(\omega_\frac{1}{n}, H)$ and $z_n(t) \subset M \backslash U$. Since $M$ is compact, by Arzel\`a-Ascoli theorem, passing to a subsequence, $z_n(t)$ converges to some $z_0(t)$ in $M \backslash U$. Because $X^{\omega_{1/n}}_H$ converges to $X^{\omega}_H$ uniformly on $M$, it is easy to check $z_0(t)$ is indeed a Hamiltonian 1-periodic orbit under $\omega$. Thus we get a contradiction. 

For the second conclusion, fix a disk $w$ spanning $x(t)$, then one gets a trivialization $\Psi: w^*TM \to D^2 \times \R^{2n}$. Restricting to the neighborhood of $\partial D^2$, one gets a symplectic path $\gamma_{\Psi}(t)$ by the relation $(\Psi \circ d\phi_{H, \omega}^t \circ \Psi^{-1})(t,\vec{v}) = (t, \gamma_{\Psi}(t) \vec{v})$. Meanwhile, the defining property of a reduced perturbation implies flow $\phi_{H, \omega}^t |_U= \phi_{H, \omega'}^t|_U$. Since Conley-Zehnder indices are computed from the same symplectic path, they have the same indices as desired. \end{proof}

Now, let us give the proof of Proposition \ref{re-sim-set-up}.

\begin{proof} [Proof of Proposition \ref{re-sim-set-up}] For each orbit $x_i(t) \in \P(\omega, H)$, take a neighborhood $U_i$ of $x_i(t)$ such that $U_i$ deformation retracts to $x_i(t)$. This can be done by taking a union of sufficiently small Darboux neighborhoods of the points on $x_i(t)$. Moreover, if it is necessary we shrink $U_i$ so that $U_i \cap U_j = \emptyset$ whenever $i \neq j$. Consider the long exact sequence 
\begin{equation} \label{les}
\cdots \rightarrow H_{dR}^2(M, \cup_i U_i; \mathbb R) \xrightarrow{\iota_*} H_{dR}^2(M; \mathbb R) \rightarrow H_{dR}^2(\cup_i U_i; \mathbb R) \rightarrow \cdots
\end{equation}
where $\iota: \Omega^2_{\tiny{\rm closed}}(M, \cup_i U_i) \to \Omega^2_{\tiny{\rm closed}}(M)$ is the inclusion. By our choice of the neighborhood $U_i$, $\cup_i U_i$ deformation retracts to $\sqcup_i x_i(t)$, so by a dimension consideration, $H_{dR}^2(\cup_i U_i; \mathbb R)=0$. Now, let
\[ H_{dR}^2(M; \mathbb R) = \bigoplus_{j=1}^m \mathbb R \cdot c_j \,\,\,\,\mbox{for some $c_j \in H^2_{dR}(M; \R)$}. \]
Then $H_{dR}^2(\cup_i U_i; \mathbb R) =0$ implies that for each basis element $c_j$, there exists some $\alpha_j \in \Omega_{\tiny{\rm closed}}^2(M, \cup_i U_i)$ such that $c_j = \iota_*[\alpha_j]$. By definition, $\alpha_j$ vanishes in every $U_i$. Meanwhile, for $\omega$ and its perturbation $\omega'$, 
\[ [\omega'] - [\omega] = \sum_{j=1}^m t_j c_j = \sum_{j=1}^m t_j \iota_*[ \alpha_j] \,\,\,\,\,\mbox{for some}\,\,\,\, (t_1, ..., t_m) \in \mathbb R^m. \]
Therefore, $[\omega_1 - \omega_0] = [\sum_{j=1}^m t_j \iota(\alpha_j) ]$. Moreover, by the definition of semi-norm $|-|_h$ on $H^*_{dR}(M; \R)$, for any $\delta>0$, there exists an exact $2$-form $d\theta$ such that 
\[ \alpha:= \sum_{j=1}^m t_j\iota(\alpha_j) + d\theta \]
and $|\alpha| \leq |[ \alpha]|_h + \delta = |[\omega' - \omega]|_h + \delta \leq |\omega' - \omega| + \delta$. We can choose $\delta$ such that $|\omega' - \omega| + \delta \leq C_1|\omega' - \omega|$ for any preferred constant $C_1 >1$, then
\begin{equation} \label{small-gap}
|\alpha| \leq C_1|\omega'- \omega|. 
\end{equation} 
Moreover, by (\ref{les}) again we know $d\theta = \iota(d\gamma)$ for some $\gamma \in \Omega^2(M, \cup_i U_i)$. Therefore, the $\alpha$ chosen above vanishes near every Hamiltonian 1-periodic orbit, that is, $\alpha \in \Omega_{\omega, H}$. 

Next, consider the homotopy $h_t = (1-t) (\omega + \alpha) + t \omega'$ where $t \in [0,1]$. Note that
\[ \left[ \frac{dh_t}{dt} \right] = [\omega' - \omega - \alpha] = [\omega' - \omega] - [ \alpha] = [\alpha] - [\alpha] = 0. \]
Therefore, $h_t$ represents the same cohomology class for each $t \in [0,1]$, then by Moser's trick, there exists some $\phi \in {\rm Diff}(M)$ such that 
\[ \phi^*\omega' = \omega + \alpha. \]
More explicitly, $\phi$ is the time-one map of the flow $\phi_t$ of vector field $X_t$ defined as 
\[ h_t(X_t, -) = - \tau \,\,\,\,\,\,\mbox{and} \,\,\,\,\, d\tau = \omega' - (\omega + \alpha).\]
By the triangle inequality, we know 
\[ |d \tau| \leq |\omega' - \omega|+ |\alpha| \leq (1+ C_1)|\omega' - \omega|. \]
Therefore using the dual norm of $\vec{\ep}$-norm on the space of vector fields, we get $|X_t| \leq C_2|\omega' - \omega|$ over $M$ for some constant $C_2$. This implies 
\[ \begin{array}{l}
d_{C^0}(\phi, \mathds{1}_M) = \sup_{x \in M} {\rm dist}(\phi(x), x) \leq C_3 |\omega' - \omega|
\end{array} \]
for some constant $C_3$. Here $C_2$ and $C_3$ involve the integration of forms and smooth vector fields along the manifold $M$.  Because $M$ is assumed to be closed, both constants are finite and only depend on $M$. Thus we get the conclusion by setting $C = C_3$. \end{proof}

\section{Energy estimations}\label{sec-e-e}
Similar to the analysis from the perturbations of Hamiltonian functions (cf. subsection \ref{ss-fcc}), the comparison between two Floer chain complexes with different symplectic structures starts from a homotopy between two symplectic structures. Explicitly, fix a symplectic structure $\omega_0$ and consider a reduced perturbation $\omega_1 = \omega_0 + \alpha$. Take a {\it smooth} cut-off function $\kappa(s)$ such that $\kappa(s) = 0$ for $s \in (-\infty, 0]$, $\kappa(s) =1$ for $s \in [1, \infty)$ and $\kappa'(s)>0$ for $s \in (0,1)$. Define an {\it interpolating homotopy $s \mapsto \omega_s$} between $\omega_0$ and $\omega_1$ by  
\begin{equation} \label{dfn-int-htp}
\omega_s = (1-\kappa(s))\omega_0 + \kappa(s) \omega_1.
\end{equation}
Note that $\omega_s$ is also a reduced perturbation of $\omega_0$ for each $s \in \R$. Now take a family of pairs $(J_s, \omega_s)$ where, for each $s \in \R$, $J_s$ is an $\omega_s$-compatible almost complex structure. They induce a family of Riemannian metrics $g_s (v,w):= \omega(v, J_sw)$. Then we can give the following definition. 
\begin{dfn} \label{def-FO} Given a Hamiltonian system $(M, \omega_0, H)$ and a reduced perturbation $\omega_1$, denote $s \mapsto \omega_s$ as an interpolating homotopy between $\omega_0$ and $\omega_1$. A parametrized Floer operator $\mathcal F^{s}$ is defined as 
\begin{equation}\label{par-FO}
\mathcal F^{s}= \frac{\partial }{\partial s} + J_s \left( \frac{\partial }{\partial t}  -X^{\omega_s}_H\right)  
\end{equation} 
where $X^{\omega_s}_H = J_s \mbox{grad}_{\omega_s} H$ and $\mbox{grad}_{\omega_s}H$ is the gradient of $H$ with respect to the induced metric $g_s$. Moreover, we call a map $u(s,t): \R \times S^1 \to M$ an {\it $\mathcal F^{s}$-trajectory} if $u$ satisfies $\mathcal F^{s}(u) = 0$. \end{dfn}

For each $\mathcal F^s$-trajectory $u$, one defines its energy by
\begin{equation*}
E(u)= \int_{-\infty}^{\infty} \int_0^1  \left| \frac{\partial u}{\partial s} \right|^2_{g_s} dtds.
\end{equation*}

\begin{dfn} \label{dfn-ad-tra}
We call an $\mathcal F^s$-trajectory {\it admissible} if it satisfies the following conditions:
\begin{itemize}
\item[(a)] The energy of $u$ is finite, that is $E(u) < \infty$;
\item[(b)] $u(s,t)$ satisfies the following asymptotic condition, where $\gamma_-(t)$ and $\gamma_+(t)$ are contractible Hamiltonian 1-periodic orbits,
\[ \lim_{s \to -\infty} u(s,t) = \gamma_-(t) \,\,\,\,\,\mbox{and} \,\,\,\,\, \lim_{s \to \infty} u(s,t) = \gamma_+(t);\]
\item[(c)] $[\gamma_+ w_+] = [\gamma_+, w_- \# u]$, where the equivalent class is defined by the relation
\[ [x, v] = [y, w] \,\,\,\,\mbox{if and only if} \,\,\,\,\,x(t) = y(t) \,\,\,\mbox{and}\,\,\, [v \# (-w)] = 0 \in H_2(M; \R). \]
\end{itemize}
\end{dfn}

Note that the symplectic action functional defined in (\ref{action-ori}) depends on symplectic structures. Along this interpolating homotopy $s \mapsto \omega_s$, the corresponding symplectic action functional at $\omega_s$ is 
\begin{equation}\label{action} 
\mathcal A_{H, \omega_s} ([\gamma,w])= - \int_{D^2} w^* \omega_s + \int_0^1H(\gamma(t),t) dt. 
\end{equation}
Since the Hamiltonian function $H$ remains the same in our perturbation discussion, for brevity, we simply denote $\mathcal A_{H, \omega_s}$ by $\mathcal A_{\omega_s}$. Our first main result in this section is the following energy estimation. Recall the notation $\alpha = \omega_1 - \omega_0$.

\begin{prop}\label{ene-est} Suppose that $u: \R \times S^1 \to M$ is an admissible $\mathcal F^s$-trajectory from $[\gamma_-, w_-]$ to $[\gamma_+, w_+]$. Then we have the energy estimation between the symplectic action functionals 
\begin{align*} \label{ene-est-in}
- (1+C|\alpha|) E(u) -  \int_{D^2} (w_-)^* \alpha & \leq  \mathcal A_{\omega_1}([\gamma_+,w_+]) - \mathcal A_{\omega_0}([\gamma_-,w_-]) \\ & \leq -(1-C|\alpha|)E(u) - \int_{D^2} (w_-)^* \alpha 
  \end{align*}
for some positive constant $C$ which is independent of $u$.
\end{prop}

Here let us state a ``local'' version of Proposition \ref{ene-est}. Given an interpolating homotopy between $\omega_0$ and $\omega_1$, for any $s<t$ in $[0,1]$, the following proposition gives an energy estimation with respect to $\omega_s$ and $\omega_t$ along this interpolating homotopy between $\omega_0$ and $\omega_1$. 

\begin{prop} \label{ene-est-cor2}
Suppose that $u: \R \times S^1 \to M$ is an admissible $\mathcal F^s$-trajectory from $[\gamma_-, w_-]$ to $[\gamma_+, w_+]$. Then we have the following energy estimation between the symplectic action functionals, 
\begin{align*}
-(1+C_{s,t} |\alpha|) E(u) + (s-t) \int_{D^2} (w_-)^*\alpha & \leq  \mathcal A_{\omega_t}([\gamma_+,w_+]) - \mathcal A_{\omega_s}([\gamma_-,w_-]) \\ 
& \leq -(1-C_{s,t} |\alpha|)E(u) + (s-t) \int_{D^2} (w_-)^*\alpha 
\end{align*}
for constant $C_{s,t} = (t-s)\cdot C$ where $C$ is the constant from Proposition \ref{ene-est}. \end{prop}

The proof of Proposition \ref{ene-est-cor2} is exactly the same as the proof of Proposition \ref{ene-est}. For simplicity, we only give the proof of Proposition \ref{ene-est}. 

\begin{proof} [Proof of Proposition \ref{ene-est}] First of all,
\begin{align*}
 \mathcal A_{\omega_1}([\gamma_+,w_+]) - \mathcal A_{\omega_0}([\gamma_-,w_-])  & = \int_{-\infty}^{\infty} \frac{d}{ds} \mathcal A_{\omega_s} ([u(s, \cdot), w_- \# u((-\infty, s] \times S^1)]) ds.
 \end{align*}
For any $s \in \R$, denote $w_s = w_- \# u((-\infty, s] \times S^1)$ and topologically it is a disk $D^2$ spanning loop $\{u(s,t)\}_{t \in S^1}$. By the definition of $\omega_s$ and symplectic action functional $\mathcal A_{\omega_s}$, 
 \begin{align*}
\mathcal A_{\omega_s} ([u(s, \cdot), w_- \# u((-\infty, s] \times S^1)]) & = - \int_{D^2} w_s^* \omega_s + \int_0^1 H(u(s,t),t)dt\\
& = - \int_{D^2} w_s^*(\omega_0 + \kappa(s) \alpha) + \int_0^1 H(u(s,t), t)dt \\
& = - \int_{D^2} w_s^* \omega_0 - \kappa(s) \int_{D^2} w_s^* \alpha + \int_0^1 H(u(s,t), t)dt.
\end{align*}
It is easy to check that for any {\it closed} 2-form $\omega$, 
\begin{equation}\label{der-2-form}
\frac{d}{ds} \int_{D^2} w_s^* \omega = \int_0^1 \omega\left(\frac{\partial u}{\partial s}, \frac{\partial u}{\partial t} \right) dt.
\end{equation}
In particular, for closed 2-forms $\omega_0$ and $\alpha$, one gets
\begin{align*}
\frac{d}{ds} \left(  - \int_{D^2} w_s^* \omega_0 - \kappa(s) \int_{D^2} w_s^* \alpha \right) & = - \int_0^1 \omega_0 \left(\frac{\partial u}{\partial s}, \frac{\partial u}{\partial t} \right) dt \\
&  - \kappa'(s) \int_{D^2} w_s^* \alpha - \kappa(s) \int_0^1 \alpha  \left(\frac{\partial u}{\partial s}, \frac{\partial u}{\partial t} \right)dt\\
& = - \kappa'(s) \int_{D^2} w_s^* \alpha - \int_0^1 \omega_s \left(\frac{\partial u}{\partial s}, \frac{\partial u}{\partial t} \right)dt.
\end{align*}
On the other hand,  
\begin{align*}
\frac{d}{ds} \int_0^1 H(u(s,t), t) dt &= \int_0^1 \frac{d}{ds} H(u(s,t),t)dt \\
& = \int_0^1 dH \left( \frac{\partial u}{\partial s} \right) dt =  -\int_0^1 \omega_s \left( X^{\omega_s}_H, \frac{\partial u}{\partial s} \right) dt.
\end{align*}
Moreover, since $u$ satisfies $\mathcal F^{\omega_s}(u) =0$, we know that  the vector field $X^{\omega_s}_H$ is $X^{\omega_s}_H= J_s(u) {\rm grad}_{\omega_s} H = J_s(u) {\textstyle \left(- \frac{\partial u}{\partial s} - J_s(u) \frac{\partial u}{\partial t} \right) }= {\textstyle \frac{\partial u}{\partial t} - J_s(u) \frac{\partial u}{\partial s}}$. Therefore, 
\begin{align*}
-\int_0^1 \omega_s \left( X^{\omega_s}_H, \frac{\partial u}{\partial s} \right) dt & = -\int_0^1 \omega_s \left(\frac{\partial u}{\partial t} - J_s(u) \frac{\partial u}{\partial s}, \frac{\partial u}{\partial s} \right) dt \\
& = \int_0^1 \omega_s \left( \frac{\partial u}{\partial s} , \frac{\partial u}{\partial t} \right)dt - \int_0^1 \left| \frac{\partial u}{\partial s} \right|^2_{g_s} dt.
\end{align*}
Combining these computations, one gets the following equality, 
\begin{equation} \label{est-mid}  \mathcal A_{\omega_1}([\gamma_+,w_+]) - \mathcal A_{\omega_0}([\gamma_-,w_-])  = -E(u) - \int_{-\infty}^{\infty} \kappa'(s) \int_{D^2} w_s^* \alpha \,ds.
\end{equation}
For the second term in (\ref{est-mid}), from an integration by parts, 
\begin{align*}
\int_{-\infty}^{\infty} \kappa'(s) \int_{D^2} w_s^* \alpha \,ds  &= \int_{-\infty}^{\infty} \frac{d}{ds} \left(\kappa(s) \int_{D^2} w_s^* \alpha\right) ds - \int_{-\infty}^{\infty} \kappa(s) \left(\frac{d}{ds}\int_{D^2} w_s^* \alpha \right)ds \\
& = \int_{D^2} (w_- \# u)^* \alpha - \int_{-\infty}^{\infty} \kappa(s) \left(\frac{d}{ds}\int_{D^2} w_s^* \alpha \right)ds \\
&  = \int_{D^2} (w_-)^* \alpha +  \int_{- \infty}^{\infty} \int_0^1(1- \kappa(s)) \alpha \left( \frac{\partial u}{\partial s}, \frac{\partial u}{\partial t}\right) dtds.
\end{align*}
Notice that the last term satisfies the following inequality,
\begin{align*} 
- \int_{- \infty}^{\infty} \int_0^1\left| \alpha \left( \frac{\partial u}{\partial s}, \frac{\partial u}{\partial t}\right)\right| dtds & \leq \int_0^1(1- \kappa(s)) \alpha \left( \frac{\partial u}{\partial s}, \frac{\partial u}{\partial t}\right) dtds \\
& \leq \int_{- \infty}^{\infty} \int_0^1\left| \alpha \left( \frac{\partial u}{\partial s}, \frac{\partial u}{\partial t}\right)\right| dtds.
\end{align*}

Now, we claim that there exist some constant $C'$ and $N$, independent of trajectory $u$ such that 
\begin{equation} \label{a-energy}
\int_{-\infty}^{\infty} \int_0^1 \left| \alpha \left(\frac{\partial u}{\partial s}, \frac{\partial u}{\partial t} \right) \right|dt ds \leq |\alpha| \cdot C' \cdot \frac{E(u)}{\sqrt{N}} + |\alpha| \cdot E(u).
\end{equation}

In fact, using the relation $X^{\omega_s}_H + J_s(u) \frac{\partial u}{\partial s}=  \frac{\partial u}{\partial t}$, what we want to estimate can be rewritten as
\begin{align*}
\int_{-\infty}^{\infty} \int_0^1\left| \alpha \left(\frac{\partial u}{\partial s}, X^{\omega_s}_H + J_s(u) \frac{\partial u}{\partial s} \right) \right|dt ds \leq &\int_{-\infty}^{\infty} \int_0^1\left| \alpha \left(\frac{\partial u}{\partial s}, X^{\omega_s}_H \right)\right|dtds \\
& + \int_{-\infty}^{\infty} \int_0^1 \left|\alpha \left(\frac{\partial u}{\partial s}, J(u) \frac{\partial u}{\partial s} \right) \right|dtds.
\end{align*}
The second term is bounded from above by $|\alpha| \cdot E(u)$. For the first term, by the asymptotic property of $u$ and the definition of $\alpha$ (which vanishes near orbits $\gamma_-(t)$ and $\gamma_+(t)$), we know that there exists some $s_u \in \R$, depending on $u$, such that $\alpha = 0$ when $s \in (-\infty, -s_u] \cup [s_u, +\infty)$. So 
\begin{equation} \label{cut off end}
\int_{-\infty}^{\infty} \int_0^1 \left| \alpha \left(\frac{\partial u}{\partial s}, X^{\omega_s}_H \right)\right| dtds  = \int_{-s_u}^{s_u} \int_0^1 \left|\alpha \left(\frac{\partial u}{\partial s}, X^{\omega_s}_H \right)\right|dtds. 
\end{equation}
Meanwhile, by Lemma 5.2 in \cite{LO95}, there exists some constant positive $N$, independent of $u$, such that for any $s_*\in [-s_u, s_u]$ (enlarge $s_u$ if necessary), 
\[ \left|\partial_s u(s, t)\right|^2_{L^2} \big|_{s =s_*}= \int_0^1  \left| \frac{\partial u} {\partial s}\right|_{s=s_*}^2 dt \geq N. \]
In other words,  if we set 
\[ \mathcal I = \left\{s_* \in \R \,\bigg| \, \left|\partial_s u(s, t)\right|^2_{L^2} \big|_{s =s_*} \geq N \right\}\]
then $[-s_u, s_u]\subset \mathcal I$. So for (\ref{cut off end}), we can improve it to be integrated over $\R/\Z \times \mathcal I$. It will not change the value of the integral by the vanishing property of $\alpha$. Therefore, we have 
\begin{align*}
\int_{-\infty}^{\infty} \int_0^1 \left| \alpha \left(\frac{\partial u}{\partial s}, X^{\omega_s}_H \right)\right| dtds & = \int_{\mathcal I} \int_0^1 \left| \alpha \left(\frac{\partial u}{\partial s}, X^{\omega_s}_H \right)\right| dtds\\
\leq |\alpha| \cdot C' \int_{\mathcal I} \int_0^1 \left| \frac{\partial u}{\partial s} \right|_{g_s} dtds
\end{align*}
where $C'$ is an upper bound of the uniform norm of vector field $X_H^{\omega_s}$ for {\it any} $s \in [0,1]$ on the closed manifold $M$. On the other hand, due to the energy constraint, (Lebesgue) measure of $\mathcal I$ satisfies $\mu(\mathcal I) \leq E(u)/N$. Applying Cauchy-Schwarz inequality, we get 
\[\left( \int_{\mathcal I} \int_0^1 1 \cdot \left| \frac {\partial u}{\partial s} \right|_{g_s} dtds \right)^2 \leq \left( \int_{\mathcal I} \int_0^1 1^2 dtds \right) \cdot \left( \int_{\mathcal I} \int_0^1 \left| \frac{\partial u}{\partial s} \right|_{g_s} ^2 dtds  \right) \leq \frac{E(u)^2}{N}. \]  
Together, we get the desired conclusion by setting constant $C : = \frac{C'}{\sqrt{N}} + 1$. \end{proof}

For later use, we need an energy estimation of another type.  

\begin{dfn} \label{dfn-sym-FO} Fix a finite number $R \in \R$. Choose a ``symmetric'' cut-off function as follows. Let $\kappa(s) = 0$ for $s \in (-\infty, -(R+1)] \cup [R+1, \infty)$ and $\kappa(s) = 1$ for $s \in [-R, R]$. Moreover, $\kappa'(s)>0$ for $s \in (-(R+1), -R)$ and $\kappa'(s)<0$ for $s \in (R, R+1)$. Using this $\kappa(s)$, we can define a {\it symmetric homotopy} $s \mapsto \omega_s$ as in (\ref{dfn-int-htp}) from $\omega_0$ to itself which passes through $\omega_1$.  Accordingly, we can define a Floer operator $\mathcal F_{\rm \rm sym}^s$ and an $\mathcal F_{\rm sym}^s$-trajectory as in Definition \ref{def-FO}, and an admissible $\mathcal F_{sym}^s$-trajectory as in Definition \ref{dfn-ad-tra}. \end{dfn}

Then we have the following energy estimation. 

\begin{prop} \label{ene-est-cor3}
Suppose that $u: \R \times S^1 \to M$ is an admissible $\mathcal F_{\rm sym}^s$-trajectory from $[\gamma_-, w_-]$ to $[\gamma_+, w_+]$. Then we have the following energy estimation between the symplectic action functionals, 
\begin{equation*} \label{ene-est-in4}
 - (1+ C|\alpha|) E(u) \leq \mathcal A_{\omega_0}([\gamma_+, w_+]) - \mathcal A_{\omega_0}([\gamma_-, w_-]) \leq -(1-C|\alpha|)E(u) 
  \end{equation*}
where the constant $C$ is the one from Proposition \ref{ene-est}. \end{prop}

\begin{proof} Denote $w_s = w_- \# u((-\infty, s] \times S^1)$. Similarly to (\ref{est-mid}), one gets
\begin{align*}
 \mathcal A_{\omega_0}([\gamma_+,w_+]) - \mathcal A_{\omega_0}([\gamma_-,w_-])  &  = - E(u) -  \int_{-\infty}^{\infty} \kappa'(s) \int_{D^2} w_s^* \alpha \,ds.
 \end{align*}
Then, integration by parts, one gets
\begin{align*}
\int_{-\infty}^{\infty} \kappa'(s) \int_{D^2} w_s^* \alpha \,ds  &= \int_{-\infty}^{\infty} \frac{d}{ds} \left(\kappa(s) \int_{D^2} w_s^* \alpha\right) ds - \int_{-\infty}^{\infty} \kappa(s) \left(\frac{d}{ds}\int_{D^2} w_s^* \alpha \right)ds \\
& = - \int_{-\infty}^{\infty} \kappa(s) \left(\frac{d}{ds}\int_{D^2} w_s^* \alpha \right)ds \\ &=  -\int_{-\infty}^{\infty} \int_0^1\kappa(s) \alpha \left( \frac{\partial u}{\partial s}, \frac{\partial u}{\partial t}\right) dtds.
\end{align*}
The last equality comes from (\ref{der-2-form}). Moreover, 
\[ -C|\alpha| E(u) \leq  \int_{-\infty}^{\infty} \int_0^1\kappa(s) \alpha \left( \frac{\partial u}{\partial s}, \frac{\partial u}{\partial t}\right) dtds \leq C|\alpha| E(u)\]
by claim (\ref{a-energy}). Therefore, we get the conclusion. \end{proof}

Observe that there is no term $\int_{D^2} w_-^* \alpha$ in the estimation from Proposition \ref{ene-est-cor3} exactly because the cut-off function $\kappa(s)$ for $\mathcal F_{sym}^s$ is symmetric. When working on a general symplectic manifold, the value of $\int_{D^2} w_-^* \alpha$ depends on the disk $w_-$ spanning $\gamma_-$, so the action on $w_-$ from $\pi_2(M)$ can possibly change this value. In other words, the estimation from Proposition \ref{ene-est} does not provide a uniform bound, which causes an algebraic difficulty involving the Novikov finiteness condition. On a different topic, when using the energy estimation from Proposition \ref{ene-est} or Proposition \ref{ene-est-cor3}, we always assume $|\alpha|$ is sufficiently small (which is equivalent to an assumption that the perturbation $\omega_1$ is sufficiently close to $\omega_0$) so that $C|\alpha|<1$.

\section{Proof of Theorem \ref{continuity result 1}} \label{sec-cont}

The main part of the proof of Theorem \ref{continuity result 1} is from the following quantitative comparison. Recall that the definition of a $\delta$-quasiequivalence is given by Definition \ref{dfn-qe}. 

\begin{lemma} \label{quasi-equ} Let $(M, \omega, H)$ be a Hamiltonian system where $M$ is aspherical, and $\omega'$ is a perturbation of $\omega$. Then there exists an $S(\alpha)$-quasiequivalence between filtered complexes  $(\CF_*(M, J, H, \omega), \partial_{\omega}, \ell_{\omega})$ and $(\CF_*(M, J, H, \omega'), \partial_{\omega'}, \ell_{\omega'})$ for some constant $S(\alpha)$ that depends on $\alpha$. \end{lemma}

Since $M$ is aspherical, Novikov field $\Lambda^{\mathcal K, \Gamma} = \K$. Therefore, both Floer chain complexes $(\CF_*(M, J, H, \omega), \partial_{\omega})$ and $(\CF_*(M, J, H, \omega'), \partial_{\omega'})$ are finite dimensional over $\K$. 

\begin{proof} Fix a basis of $\CF_*(M, J, H, \omega)$ over $\K$, say $\{[x_1, w_1], ..., [x_n, w_n]\}$. They are also a basis for $\CF_*(M, J, H, \omega')$ by Proposition \ref{no-other-orbits}. We need to find a quadruple $(\Phi, \Psi, K, K')$ that satisfies the conditions in Definition \ref{dfn-qe}. 

First, fix any interpolating homotopy between $\omega$ and $\omega'$ as in (\ref{dfn-int-htp}). Recall that the Floer operator $\mathcal F^s$ and the admissible $\mathcal F^s$-trajectories are defined in Definition \ref{def-FO} and \ref{dfn-ad-tra}. Consider a map $\Phi: \CF_*(M, J, H, \omega) \to \CF_*(M, J, H, \omega')$ defined by 
\begin{equation} \label{chain map one}
\Phi([x_i, w_i]) = \sum_{j \in\{1, ..., n\},  \, \mu_{CZ}([x_i,w_i]) = \mu_{CZ}([x_j,w_j])} n([x_i, w_i], [x_j, w_j]) [x_j, w_j] 
\end{equation}
where number $n([x_i, w_i], [x_j, w_j])$ is defined by counting admissible $\mathcal F^s$-trajectories connecting $[x_i,w_i]$ and $[x_j, w_j]$. Due to the index condition, $n([x_i, w_i], [x_j, w_j])$ is a finite number. Moreover, since there are only finitely many Hamiltonian 1-periodic orbits, the sum in the expression (\ref{chain map one}) is a finite sum. Finally, by the standard Floer gluing argument (cf. subsection 3.3 and subsection 3.4  in \cite{Sal97}), $\Phi$ is a chain map. Similarly, we can define $\Psi: \CF_*(M, J, H, \omega') \to \CF_*(M, J, H, \omega)$ and it is also a chain map. 

Now let us study the change of filtrations. By Proposition \ref{ene-est}, any admissible $\mathcal F^s$-trajectory $u$ connecting $[x_i,w_i]$ and $[x_j, w_j]$ gives rise to an inequality
\begin{equation} \label{upperbound} 
\mathcal A_{\omega'}([x_j, w_j]) - \mathcal A_{\omega}([x_i, w_i]) \leq -(1-C|\alpha|)E(u) - \int_{D^2} w^*_i \alpha \leq -\int_{D^2} w^*_i \alpha
\end{equation}
because we always assume $\omega'$ is sufficiently close to $\omega$ (so $1-C|\alpha|>0$). By the definition of filtration function (\ref{dfn-filtration}), for any chain $c \in \CF_*(M, J, H, \omega)$, there exists some $j_0 \in \{1, ..., n\}$ {\it depending on $\alpha$} such that $\ell_{\omega'}(\Phi(c)) = \mathcal A_{\omega'}([x_{j_0}, w_{j_0}])$. Meanwhile, there exists some $i_0 \in \{1, ..., n\}$ such that $[x_{i_0}, w_{i_0}]$ from the chain $c$ connects to $[x_{j_0}, w_{j_0}]$ by an admissible $\mathcal F^s$-trajectory $u$. Therefore, 
\begin{equation} \label{ub} \mathcal A_{\omega'}([x_{j_0}, w_{j_0}]) - \mathcal A_{\omega}([x_{i_0}, w_{i_0}]) \leq -\int_{D^2} w^*_{i_0} \alpha. 
\end{equation}
Since $\ell_{\omega}(c) \geq \mathcal A_{\omega}([x_{i_0}, w_{i_0}])$, we have $\ell_{\omega'}(\Phi(c)) - \ell_{\omega}(c) \leq -\int_{D^2} w^*_{i_0} \alpha$. Denote 
\[ S(\alpha) = \max_{i} \left|\int_{D^2} w_i^* \alpha\right|,\]
and it is non-negative. Then $\ell_{\omega'}(\Phi(c)) \leq \ell_{\omega}(c) + S(\alpha)$ for any $c \in \CF_*(M, J, H, \omega)$. A similar argument for $\Psi$ results in the same constant $S(\alpha)$ and a similar filtration inequality. 

Next, we construct a homotopy between $\mathds{1}_{\omega}$ and $\Psi \circ \Phi$ where $\mathds{1}_{\omega}$ is the identity map on the complex $(\CF_*(M, J, H, \omega), \partial_{\omega})$. So far, we have two different homotopies of symplectic structure $\omega$: one is the constant homotopy of $\omega$ that corresponds to $\mathds{1}_{\omega}$, and the other is a symmetric homotopy of $\omega$ which passes through $\omega'$. This homotopy is induced by the composition $\Psi \circ \Phi$. Let us denote this symmetric homotopy as $s \mapsto \omega_s$. Take a homotopy $\{\omega_{s, \lambda}\}_{\lambda \in [0,1]}$ between these two homotopies $s \mapsto \omega_s$ and $\omega$ such that $\omega_{s,0} = \omega$ and $\omega_{s,1} = \omega_s$. Moreover, we require $\omega_{s,\lambda}$ to be a symmetric homotopy for each $\lambda \in [0,1]$. Recall that the operator $\mathcal F^{s,\lambda}_{\rm sym}$ and the admissible $\mathcal F^{s,\lambda}_{\rm sym}$-trajectory with respect to $\omega_{s,\lambda}$ are defined in Definition \ref{dfn-sym-FO}. Define $K: \CF_*(M,J, H, \omega) \to \CF_{*+1}(M, J, H, \omega)$ as 
\begin{equation} \label{htp-map-one}
K([x_i, w_i]) = \sum_{j \in\{1, ..., n\},\,\mu_{CZ}([x_i,w_i])+1 = \mu_{CZ}([x_j, w_j])} n([x_i, w_i], [x_j, w_j]) [x_j, w_j] 
\end{equation}
where number $n([x_i, w_i], [x_j, w_j])$ is defined by counting pairs $(u, \lambda)$ where $u$ is an admissible $\mathcal F^{s,\lambda}_{\rm sym}$-trajectory connecting $[x_i,w_i]$ and $[x_j, w_j]$. Again, $n([x_i, w_i], [x_j, w_j])$ is finite and (\ref{htp-map-one}) is a finite sum. Moreover, by a standard argument considering corresponding moduli space, one can show that $K$ provides a homotopy between $\mathds{1}_{\omega}$ and $\Psi \circ \Phi$. Similarly, we obtain a homotopy $K'$ between $\mathds{1}_{\omega'}$ and $\Phi \circ \Psi$ where $\mathds{1}_{\omega'}$ is the identity map on the complex $(\CF_*(M, J, H, \omega'), \partial_{\omega'})$.

Again, let us study the changes of filtrations. For any chain $c \in \CF_*(M, J, H, \omega)$, suppose $\ell_{\omega}(K(c)) = \mathcal A_{\omega}([x_{q_0}, w_{q_0}])$ for some $q_0 \in \{1, ..., n\}$. There exists some $[x_{p_0}, w_{p_0}]$ from chain $c$ connected to $[x_{q_0}, w_{q_0}]$ by an admissible $\mathcal F^{s,\lambda}$-trajectory for some $\lambda \in [0,1]$. By Proposition \ref{ene-est-cor3}, $\mathcal A_{\omega}([x_{q_0}, w_{q_0}]) \leq \mathcal A_{\omega}([x_{p_0}, w_{p_0}])$. This implies 
\[ \ell_{\omega}(K(c)) - \ell_{\omega}(c) \leq \mathcal A_{\omega}([x_{q_0}, w_{q_0}]) -\mathcal A_{\omega}([x_{p_0}, w_{p_0}]) \leq 0.\]
Since $S(\alpha)  \geq 0$, in particular, $\ell_{\omega}(K(c)) \leq \ell_{\omega}(c) + 2 S(\alpha)$ for any $c \in \CF_*(M, J, H, \omega)$. A similar inequality holds for the other homotopy $K'$. Thus we get the conclusion.\end{proof}

\begin{remark} Under the hypothesis of Lemma \ref{quasi-equ}, the same argument can prove the following more general result if one uses Proposition \ref{ene-est-cor2}: for any $s\leq t$ in $[0,1]$, there exists some constant $S_{s,t}(\alpha)$ such that $(\CF_*(M, J, H, \omega_s), \partial_{\omega_s}, \ell_{\omega_s})$ and $(\CF_*(M, J, H, \omega_t), \partial_{\omega_t}, \ell_{\omega_t})$ are $S_{s,t}(\alpha)$-quasiequivalent. Moreover, $S_{s,t}(\alpha) \leq |s-t|\cdot C|\alpha|$ for some constant $C$. \end{remark}

Now, we are ready to give the proof of Theorem \ref{continuity result 1}. 

\begin{proof} [Proof of Theorem \ref{continuity result 1}]
From Proposition \ref{re-sim-set-up}, if $\omega'$ is a perturbation of $\omega$, then there exists a $\phi \in {\rm Diff}(M)$ such that $\phi^*\omega'$ is a reduced perturbation of $\omega$. As elaborated in Section \ref{simple}, if we replace $\omega'$ with $\phi^*\omega$ in our discussion, then changes from the perturbation of Hamiltonian functions (from $H$ to $\phi^*H$) needs to be considered. Theorem \ref{Lip} implies that the resulting changes on both spectral invariants and boundary depth are no greater than $|\phi^*H - H|_{H}$. Since $M$ is compact,  $|\phi^*H - H|_{H} \leq C' |\omega' - \omega|$ for some constant $C'$ that involves a $C^0$-norm of $H$. In what follows, we will focus on the case when $\omega'$ is a reduced perturbation of $\omega$. 

The desired conclusion for boundary depth directly comes from Lemma \ref{quasi-equ} and Theorem \ref{bd-qe}. For spectral invariants, by the same idea as the proof of (iii) in Theorem 3.1 in \cite{PSS96}, we have the following commutative diagram 
\[ 
\xymatrix{
 & H_*(M; \mathcal K) \ar[dl]_{PSS^{\omega}_*} \ar[dr]^{PSS^{\omega'}_*} \\
\HF_*(M, J, H, \omega) \ar[rr]_{\Phi_*} && \HF_*(M, J, H, \omega')}
\]
where $\Phi_*$ is the chain map constructed from Lemma \ref{quasi-equ}. By Theorem \ref{realization thm}, there exists some element $c \in \CF_*(M, J, H, \omega)$ such that $\rho(a, H; \omega) = \ell_{\omega}(c)$, and $[c] = PSS_*^{\omega}(a)$. Then 
\[ [\Phi(c)] = \Phi_*([c]) = \Phi_* (PSS_*^{\omega}(a)) = PSS_*^{\omega'}(a). \]
Therefore, 
\[\rho(a, H; \omega') - \rho(a, H; \omega) \leq \ell_{\omega'}(\Phi(c)) - \ell_{\omega}(c) \leq S(\omega'-\omega). \]
Switch the role of $\omega$ and $\omega'$, then we get the other inequality. Therefore, 
\[ |\rho(a, H; \omega) - \rho(a, H; \omega')| \leq S(\omega' - \omega).\]
Finally, it is easy to see that there exists a constant $C$ such that $S(\omega' - \omega) \leq C|\omega' - \omega|$. Thus we get the conclusion. 
\end{proof}

To end this section, we will give the proof of Corollary \ref{continuity result 2}. 

\begin{proof} [Proof of Corollary \ref{continuity result 2}] Similarly to the argument in the proof of Theorem \ref{continuity result 1}, we will only focus on the case when $\omega'$ is a reduced perturbation of $\omega$. By the hypothesis on the dimension of $H^2(M; \K)$, there exists a number $\lambda$ (sufficiently) close to $1$ such that $[\omega']= \lambda [\omega]$. A key observation is that if we rescale $\lambda H$ and $\lambda \omega$, then  
\[ \beta(\lambda H; \lambda \omega)  = \lambda \beta(H; \omega) \,\,\,\,\mbox{and}\,\,\,\, \rho([M], \lambda H; \lambda \omega) = \lambda \rho([M], H; \omega). \]
Hence, by (b) in Theorem \ref{Lip} and Theorem \ref{continuity result 1},
\begin{align*}
|\beta(H; \omega) - \beta(H; \omega')| & \leq |\beta(H; \omega) - \beta(\lambda H; \lambda \omega)| \\
&+ |\beta(\lambda H; \lambda \omega) - \beta(H; \lambda \omega)| \\
& + |\beta(H; \lambda \omega) - \beta(H; \omega')|\\
& \leq |1-\lambda| \beta(H; \omega) + |1-\lambda| ||H||_H + C|\lambda \omega - \omega'|\\
& \leq |1-\lambda| \beta(H; \omega) + |1-\lambda| ||H||_H + C(|1-\lambda||\omega| + |\omega - \omega'|) \\
& = |1-\lambda| \cdot A + C|\omega - \omega'|
\end{align*}
where $A : =\beta(H; \omega) + ||H||_H + C|\omega|$ and $C$ is the constant from Theorem \ref{continuity result 1}. Moreover, $[\omega - \omega'] = (1-\lambda)[\omega]$ implies 
\begin{equation} \label{small-est}
|1-\lambda||[\omega]|_h = |[\omega - \omega']|_h \leq |\omega - \omega'|. 
\end{equation}
Therefore, $|1-\lambda| \leq \frac{1}{|[\omega]|_h} |\omega - \omega'|$. So 
\[ |1-\lambda| \cdot A + C|\omega - \omega'| \leq \left(\frac{A}{|[\omega]|_h} + C\right) \cdot |\omega - \omega'|. \]
Therefore, set $\tilde C_2 = \frac{A}{|[\omega]|_h} + C$ and we get the conclusion. 

Similarly, for spectral invariants, by (a) in Theorem \ref{Lip} and Theorem \ref{continuity result 1},
\begin{align*}
|\rho([M], H; \omega) - \rho([M], H; \omega')| & \leq |\rho([M], H; \omega) - \rho([M], \lambda H; \lambda \omega)| \\
&+ | \rho([M], \lambda H; \lambda \omega) -  \rho([M], H; \lambda \omega)| \\
& + |\rho([M], H; \lambda \omega) - \rho([M], H; \omega')|\\
& \leq |1-\lambda| \rho([M], H; \omega)  + |1-\lambda| ||H||_H + C|\lambda \omega - \omega'|\\
& \leq |1-\lambda| \cdot B + C|\omega - \omega'|
\end{align*}
where $B := \rho([M], H; \omega) + ||H||_H + C|\omega|$ and $C$ is the constant from Theorem \ref{continuity result 1}. The same estimation as in (\ref{small-est}) implies the desired conclusion when we set $\tilde{C}_1 = \frac{B}{|[\omega]|_h} + C$. \end{proof}

\begin{remark}
From the perspective of persistent homology theory in \cite{UZ16}, we can associate barcodes, denoted by $\mathcal B_{\omega}$ and $\mathcal B_{\omega'}$, to both $\mathcal C_{\omega} := (\CF_*(M, J, H, \omega), \partial_{\omega}, \ell_{\omega})$ and $\mathcal C_{\omega'} := (\CF_*(M, J, H, \omega'), \partial_{\omega'}, \ell_{\omega'})$, respectively. Lemma \ref{quasi-equ} implies that their quasiequivalence distance (Definition 8.1 in \cite{UZ16}) satisfies 
\[ d_Q(\mathcal C_{\omega}, \mathcal C_{\omega'}) \leq \frac{S(\omega - \omega')}{2}.\]
Recall that $d_B$ denotes the bottleneck distance (Definition 8.14 in \cite{UZ16}). Then Stability Theorem in \cite{UZ16} implies the following inequality, 
\[ d_B(\mathcal C_{\omega}, \mathcal C_{\omega'}) \leq 2 d_Q(\mathcal C_{\omega}, \mathcal C_{\omega'}) \leq S(\omega - \omega')\leq C|\omega - \omega'|. \]
This can be regarded as a generalization of Theorem \ref{continuity result 1}. \end{remark} 

\section{Proofs of Theorem \ref{bscP} and Theorem \ref{bd-cap}} \label{sec-app}

\begin{proof} [Proof of Theorem \ref{bscP}] Let $U\subset \Sigma$ be a disjoint union of simply connected open subsets and $H \in C^{\infty}(\R/\Z \times U)$. Pick an open ball $V \subset \Sigma \backslash \bar{U}$ and fix a closed 2-form $\alpha$ positively supported in $V$ and that vanishes elsewhere. By our choice of $H$, $\alpha \in \Omega_{\omega, H}$. Fix $\lambda \geq 1$. Consider the following isotopy of the reduced perturbations,
\[ \omega_s = \omega + s (\lambda \alpha)\]
where $s \in [0,1]$. Take $\lambda$ sufficient large such that, on $(\Sigma, \omega_1)$, $U$ can be viewed as a disjoint union of topological balls with total area sufficiently smaller than the total area of $\Sigma$. Then it is displaceable. By an energy-capacity inequality, say Corollary 3.3 in \cite{Ush10}, 
\begin{equation} \label{ece}
\rho(a, H, \omega_1) \leq e^{\omega_1}(U) <+\infty
\end{equation}
where $e^{\omega_1}(U)$ denotes the displacement energy of $U$ under symplectic structure $\omega_1$. Moreover since $\omega_s$ coincides  with $\omega$ inside $U$ and $H$ vanishes outside $U$, 
\[ \rho(a, H; \omega_s) \in \S(H, \omega) \,\,\,\,\,\,\,\,\mbox{where} \,\,\,\,\,\ s \in [0,1]\]
for any $a \in H_*(\Sigma, \K)$. Note that in our situation $\S(H, \omega_0)$ is just a finite set of $\R$. Then Theorem \ref{continuity result 1} implies that $\{\rho(a, H; \omega_s)\}_{0 \leq t \leq 1}$ is a continuous path over a finite set. So $\rho(a, H; \omega_s)$ is constant for all $s \in [0,1]$ and then $\rho(a, H; \omega) \leq e^{\omega_1}(U)$. We get the desired conclusion by setting $K = e^{\omega_1}(U)$ in Definition \ref{bsc}. \end{proof}

\begin{remark} (i) Corollary 3.3 in \cite{Ush10} only claims an energy-capacity inequality (\ref{ece}) for $a = [\Sigma]$. However, since $[\Sigma]$ is a unit under quantum product $\ast$ (here it is just intersection of homology classes), the triangle equality of spectral invariants says
\[ \rho(a, H; \omega_1)  = \rho(a * [\Sigma], 0+H; \omega_1) \leq \rho(a, 0; \omega_1) + \rho([\Sigma], H; \omega_1) \leq e^{\omega_1}(U). \]
(ii) Theorem \ref{continuity result 1} only applies locally, so in order to claim that $\{\rho(a, H; \omega_s)\}_{0 \leq t \leq 1}$ is a continuous path, we need to apply Theorem \ref{continuity result 1} inductively. To this end, we can choose a uniform size of the neighborhoods of symplectic structures such that the energy estimation, Theorem \ref{ene-est}, applies. Then our claim directly follows from the compactness of interval $[0,1]$. (iii) For $\Sigma = S^2$, the argument in the proof of Theorem \ref{bscP} does not apply. Indeed, $s \to \rho([M], H, \omega_s)$ does not provide a function from $[0,1]$  to $\S(H, \omega)$ since the corresponding fundamental classes actually lie in different quantum homologies when $\omega_s$ changes.
\end{remark}

Before giving the proof of Corollary \ref{bscPP}, we need to recall equivalent definitions of a heavy subset and a superheavy subset. This is stated as Proposition 4.1 in \cite{EP09}. 

\begin{dfn} \label{hsh1}  Given a closed subset $X \subset M$ and $a \in \QH_*(M, \omega)$, if $\zeta_a(H; \omega) = 0$ for any $H \in C^{\infty}(M)$ such that $H \leq 0$ and $H|_X =0$, then $X$ is called $a$-heavy; if $\zeta_a(H; \omega) = 0$ for any $H \in C^{\infty}(M)$ such that $H \geq 0$ and $H|_X =0$, then $X$ is called $a$-superheavy.
\end{dfn}

\begin{proof} [Proof of Corollary \ref{bscPP}] For (a), this is immediate from Definition \ref{hsh1}. In fact, for any $H \geq 0$ with $H|_X =0$, it is supported in $\Sigma_g \backslash X$ which is a disjoint union of simply connected regions. By Theorem \ref{bscP}, $\rho(a, H, \omega) \leq K$ for some finite $K \geq 0$, so $\rho(a, kH, \omega)<K$ for any $k \in \N$. By the definition of (partial) symplectic quasi-states in Definition \ref{dfn-pqs}, we know $\zeta_a(H; \omega) = 0$. Therefore, $X$ is $a$-superheavy. For (b), let $U$ be the simply connected region that $X$ lie in. By Theorem \ref{bscP} and argument above, for any $H$ supported in $U$, $\zeta_a(H; \omega) =0$. Meanwhile, we can always choose $H$ such that $H(x) \geq \delta >0$ for every $x \in X$. Then Definition \ref{hsh2} says that $X$ is not $a$-heavy. \end{proof}

Finally, we will study the continuity of some symplectic capacities. Let $A \subset M$. Recall the definitions of $c^{\omega}_{HZ}(A)$ and $c^{\omega}_{\rho}(A)$. 

\begin{dfn}\label{dfn-hz} Hofer-Zehnder capacity $c_{HZ}^{\omega}(A)$ is defined as 
\[ c_{HZ}^{\omega}(A) = \sup\{ \max H\,| \, H \in \mathcal H(A) \,\,\mbox{is HZ-admissible}\}.\]
Here $\mathcal H(A)$ contains all the autonomous functions on $M$ with compact support in $A$ and $H^{-1}(0)$ and $H^{-1}(\max(H))$ contain nonempty open sets. Moreover, HZ-admissible means that the Hamiltonian flow of $H$ under $\omega$ contains no nonconstant periodic orbit of period at most $1$. Roughly speaking, it excludes those ``fast'' Hamiltonian  orbits. \end{dfn}

\begin{dfn}\label{dfn-sc} Spectral capacity $c_{\rho}^{\omega}(A)$ is defined as 
\[ c_{\rho}^{\omega}(A) = \sup\{\rho([M], H; \omega)\,| \, H \in C_c^{\infty}(\R/\Z \times A)\}.\]
\end{dfn}

Now, let us give the proof of Theorem \ref{bd-cap}. 

\begin{proof} [Proof of Theorem \ref{bd-cap}]
Fix $\ep>0$. By definition, there exists some $H \in C_c^{\infty}(\R/\Z \times A)$ such that $c_{\rho}^{\omega}(A) < \rho([M], H; \omega) + \ep/2$. Under our hypothesis, Theorem \ref{continuity result 1} and Corollary \ref{continuity result 2} state that, near $\omega$, $\omega' \to \rho([M], H; \omega')$ is continuous when $M$ satisfies topological hypothesis in our assumption. In particular, it is lower semicontinuous. Then for this fixed $\ep$, there exists a neighborhood of $\omega$ in $\Omega_{\omega, H}$ (so in $\Omega_{\tiny \rm closed}^2(M)$) denoted by $U_{\omega}(\ep)$ such that  $\rho([M], H; \omega') - \rho([M], H; \omega) \geq -\ep/2$ for any $\omega' \in U_{\omega}(\ep)$. Hence,
\[ c_{\rho}^{\omega'}(A) - c_{\rho}^{\omega}(A) \geq  \rho([M], H; \omega') - \rho([M], H; \omega) - \ep/2  \geq -\ep. \]
Therefore, by (\ref{e-c-i}), 
\[ e^{\omega'}(A) \geq c_{\rho}^{\omega'}(A) \geq c_{\rho}^{\omega}(A) - \ep \geq c_{HZ}^{\omega}(A) - \ep.\]
Thus we get the conclusion. 
\end{proof}

\begin{remark} We can also use another capacity $c^{\omega}_{\beta}(A)$ from \cite{Ush13}. It is defined as 
\[ c_{\beta}^{\omega}(A) = \sup\{\beta(H; \omega) \,| \, H \in C^{\infty}_c(\R/\Z \times A)\}. \]
Corollary 5.12 in \cite{Ush13} says that $e^{\omega}(A) \geq \frac{1}{2} c_{\beta}^{\omega}(A)$. By the lower semicontinuity of boundary depth from Theorem \ref{continuity result 1} and Corollary \ref{continuity result 2}, we can prove that $e^{\omega'}(A)$ is bounded from below by $\frac{1}{2} c^{\omega}_{\beta}(A) - \ep$ for any $\omega' \in U_{\omega}(\ep)$. \end{remark}

\section{Variant Floer chain complexes} \label{vari-fcc}
\subsection{Novikov ring with multi-finiteness condition} \label{vat-nov} 
In this subsection, we will study the Novikov ring with multi-finiteness condition defined in Definition \ref{dfn-nov-mfc}. As explained in the introduction, this is the starting point of comparing Floer chain complexes with different symplectic structures in a general set-up. First of all, let us recall the following extended version of Novikov ring considered in \cite{Ush08},
\begin{equation*} \label{exp1}
\Lambda_{\omega} = \left\{ \sum_{A \in H^S_2(M)} a_A T^A\, \bigg| a_A \in \mathcal K, (\forall C \in \R) \,(\#\{a_A \neq 0 \, | \, [\omega](A) \leq C\} < \infty) \right\}
\end{equation*} 
where $H_2^S(M)$ is the image of $\pi_2(M)$ in $H_2(M;\Z)/{\rm Tor}$ under Hurewicz map $\iota: \pi_2(M) \to H_2(M;\Z)$. By the following exact sequence, 
\begin{equation} \label{ses}
0 \to \ker[\omega] \to H^S_2(M) \xrightarrow{[\omega]} \Gamma_{\omega} \to 0, 
\end{equation}
where $\Gamma_{\omega} = {\rm Im}[\omega]$, we can write any element $x = \sum_{A \in H_2^S(M)} a_A T^A \in \Lambda_{\omega}$ as 
\begin{equation} \label{iso}
x = \sum_{g \in \Gamma} a_g T^g \,\,\,\,\mbox{where} \,\,\, a_g \in \K[\ker[\omega]].
\end{equation} 
In other words, $\Lambda_{\omega}$ can be rewritten as  
\begin{equation} \label{exp2}
\Lambda_{\omega} = \left\{ \sum_{g \in \Gamma_{\omega}} a_g T^g \, \bigg| a_g \in \mathcal K[\ker[\omega]], (\forall C \in \R) \,(\#\{a_{g} \neq 0 \, | \, g \leq C\} < \infty) \right\}.
\end{equation}
Note that in general, $\mathcal K[\ker[\omega]]$ is not necessarily a PID, therefore, by Theorem 4.2 in \cite{HS95}, $\Lambda_{\omega}$ is not necessarily a PID. However, since $\mathcal K$ is Noetherian, $\mathcal K[\ker[\omega]]$ is Noetherian. Compared with the very often used Novikov field $\Lambda^{\mathcal K, \Gamma}$ defined in (\ref{dfn-nf}), there is a natural homomorphism $R_{\omega}: \Lambda_{\omega} \to \Lambda^{\mathcal K, \Gamma_{\omega}}$ defined by 
\begin{equation} \label{trivial-ker}
\sum_{g \in \Gamma_{\omega}} a_g T^g \xrightarrow{R_{\omega}} \sum_{g \in \Gamma_{\omega}}[a_g] T^g
\end{equation}
where $[a_g] \in \K$ is defined as follows: if $a_g = \sum_h a_{g,h} S^h$ where $h \in \ker[\omega]$ and $S$ is the formal variable, then $[a_g] = \sum_h a_{g,h}$. In other words, we uniformly weight any $h \in \ker[\omega]$ by the value zero. Then $\Lambda^{\mathcal K, \Gamma_{\omega}}$ can be regarded as a $\Lambda_{\omega}$-module. The following property of $\Lambda_{\omega}$ will be useful later.

\begin{lemma} \label{int-domain}
$\Lambda_{\omega}$ is an integral domain. 
\end{lemma}
\begin{proof} First, because $\ker[\omega]$ is a subgroup of $H^S_2(M)$ which is torsion-free and abelian, $\ker[\omega]$ is also torsion-free and abelian. By Proposition 1.3 and Lemma 0.1 in \cite{Coh74}, $\mathcal K[\ker[\omega]]$ is an integral domain. Take two non-zero elements $\lambda_1$ and $\lambda_2$ in $\Lambda_{\omega}$ and write them as 
\[ \lambda_i  = \sum_{g_{ij} \in \Gamma_{\omega}} a_{g_{ij}}  T^{g_{ij}}  \,\,\,\,\,\,\,\,\, \mbox{where} \,\,\,\,\,\,\, a_{g_{ij}} \in \mathcal K[\ker[\omega]] \,\,\,\,\mbox{and}\,\,\,\,i = 1,2. \]
Because $\lambda_i \neq 0$, the finiteness condition implies that there exist smallest powers for both $i=1,2$. Denote these smallest powers by $g_{1j_1}$ for $\lambda_1$ (for some $j_1$) and $g_{2 j_2}$ for $\lambda_2$ (by some $j_2$). Their corresponding coefficients in $\mathcal K[\ker[\omega]]$, $a_{g_{1j_1}}$ and $a_{g_{2j_2}}$, are in particular non-zero. Then  $a_{g_{1j_1}} \cdot a_{g_{2j_2}}  \neq 0$ implies $\lambda_1 \cdot \lambda_2 \neq 0$. Therefore, $\Lambda_{\omega}$ is an integral domain. \end{proof} 

Recall the construction in subsection \ref{ss-sgm}. There exist homology classes $[\omega_1], ..., [\omega_m]$ in $H^2(M; \K)$ forming a polygon $\Delta(\omega)$ containing $[\omega]$ inside such that any reduced perturbation $\omega'$ (sufficiently close to $\omega$) has $[\omega']$ written as a convex linear combination as (\ref{cov-lin}). Set $\omega_0 = \omega$. Observe that any such convex linear combination can be inductively constructed from the following two-term case,
\[ [\omega'] = (1-t)[\omega_0] + t [\omega_1].\]
Here, we require $[\omega_0]$ and $[\omega_1]$ to be linearly independent over $\K$ in $H^2(M; \K)$. Otherwise, it will be reduced to the rescaling case as studied in Corollary \ref{continuity result 2}. This automatically requires that $\dim_{\K} H^2(M; \K) \geq 2$. For simplicity, we will mainly consider this two-term case in the rest of the paper. Moreover, note that $\omega'$ and $(1-t)\omega_0 + t \omega_1$ are not necessarily the same. Instead, they differ by an exact 2-form. By Moser's trick, it is easy to deal with the perturbations from exact 2-forms. Therefore, without loss of generality, we will also assume $\omega' = (1-t) \omega_0 + t \omega_1$ for some $t \in [0,1]$. 

\begin{dfn} \label{dfn-multi} Fix $\omega_0$ and its reduced perturbation $\omega_1$ such that $[\omega_0]$ and $[\omega_1]$ are linearly independent over $\K$ in $H^2(M; \K)$. We call the following $\Lambda_{[0,1]}$ a {\it Novikov ring with multi-finiteness condition}, 
{\small \begin{equation*} \label{dfn-nr-multi}
\Lambda_{[0,1]} = \left\{ \sum_{A \in H^S_2(M)} a_A T^A \, \bigg| \, a_A \in \mathcal K, (\forall C \in \R)\,(\forall t \in [0,1])\, (\#\{a_A \neq 0 \, | \, [\omega_t](A) \leq C\} < \infty) \right\}.
\end{equation*} }
\end{dfn} 

By this definition, $\Lambda_{[0,1]} = \bigcap_{t \in [0,1]} \Lambda_{\omega_t}$. Two immediate properties of $\Lambda_{[0,1]}$ follow. 
\begin{itemize}
\item[(a)] $\Lambda_{[0,1]}$ is non-empty because every finite length power series, that is polynomial, lies inside. 
\item[(b)] $\Lambda_{[0,1]}$ is also an integral domain because $\Lambda_{[0,1]}$ is a subring of $\Lambda_{\omega_t}$ for any $t \in [0,1]$ and any subring of an integral domain is also an integral domain (cf. Lemma \ref{int-domain}). 
\end{itemize}

In fact, instead of considering uncountably many $\omega_t$ for the finiteness condition in Definition \ref{dfn-multi}, the following result shows that $\Lambda_{[0,1]}$ has a much easier structure. 

\begin{lemma} \label{int-prop} $\Lambda_{[0,1]} = \Lambda_{\omega_0} \cap \Lambda_{\omega_1}$. \end{lemma}

\begin{proof} The inclusion $\Lambda_{[0,1]} \subset \Lambda_{\omega_0} \cap \Lambda_{\omega_1}$ is trivial since the finiteness condition in (\ref{dfn-nr-multi}) is in particular valid for $t=0$ and $t=1$. Now we prove the other inclusion. Take any $x \in \Lambda_{\omega_0} \cap \Lambda_{\omega_1}$, say $x = \sum_{A \in H^S_2(M)} a_A T^A$, that satisfies the finiteness conditions for both $\omega_0$ and $\omega_1$, then for any $C \in \R$ and for any $t \in (0,1)$, $[\omega_t](A) = (1-t) [\omega_0](A) + t [\omega_1](A) \leq C$ implies that either $(1-t)[\omega_0](A) \leq C/2$ or $t [\omega_1](A) \leq C/2$. Therefore, 
\[ [\omega_0](A) \leq \frac{C}{2(1-t)} \,\,\,\,\, \mbox{or} \,\,\,\,\, [\omega_1](A) \leq \frac{C}{2t}. \] 
The defining property of element $x$ implies that in either case, there are only finitely many $A$'s. Therefore, $x$ also satisfies the finiteness condition in (\ref{dfn-nr-multi}).\end{proof}

\begin{ex} \label{counter-ex} Take $x = \sum_{n=0}^{\infty} T^{nA}$ where $A \in \ker[\omega_1]\backslash \ker[\omega_0]$. This is an element in $\Lambda_{\omega_0}$ but not in $\Lambda_{[0,1]}$. In general, $\Lambda_{[0,1]}$ is strictly contained in $\Lambda_{\omega_t}$ for any $t \in [0,1]$. Moreover, the following computation shows that $\Lambda_{[0,1]}$ does not act on $\Lambda_{\omega_0} \backslash \Lambda_{[0,1]}$. Take $x$ as above, 
\[ (1 - T^{A})x = (1- T^A) \left(\sum_{n=0}^{\infty} T^{nA}\right) = 1 \in \Lambda_{[0,1]}. \]
This is different from the observation in the proof of Theorem 2.5 in \cite{Ush08} that for any proper subgroup $G \leq \Gamma_{\omega}$, $\Lambda^{\K, G}$ acts on $\Lambda^{\K, \Gamma_{\omega} \backslash G}$. \end{ex}

\begin{remark} Definition \ref{dfn-multi} can be easily generalized to higher dimensional situations. Assume that $\Delta(\omega)$ is a polygon in $H^2(M; \K)$ that contains $[\omega]$ and also has vertices $[\omega_1], ..., [\omega_m]$. Define
{\small \begin{equation*} \label{high-dim-nr}
\Lambda_{\Delta(\omega)}= \left\{ \sum_{A \in H^S_2(M)} a_A T^A \, \bigg| \, a_A \in \mathcal K, (\forall C \in \R)\,(\forall [\omega'] \in \Delta(\omega)) \, (\#\{a_A \neq 0 \, | \, [\omega'](A) \leq C\} < \infty) \right\}. 
\end{equation*}}
In particular, $\Lambda_{\Delta(\omega)}$ is an integral domain. Moreover, $\Lambda_{\Delta(\omega)} = \Lambda_{\omega} \cap \Lambda_{\omega_1} \cap \ldots \cap \Lambda_{\omega_m}$. \end{remark}

Since we have inclusions $i_0: \Lambda_{[0,1]} \to \Lambda_{\omega_0}$ and $i_1: \Lambda_{[0,1]} \to \Lambda_{\omega_1}$, both $\Lambda_{\omega_0}$ and $\Lambda_{\omega_1}$ are $\Lambda_{[0,1]}$-modules. Together with (\ref{trivial-ker}), we have the following coefficient extensions, 
\[ \xymatrix{
     & \Lambda_{\omega_0} \ar[r]^-{R_{\omega_0}} & \Lambda^{\mathcal K, \Gamma_{\omega_0}} \\
     \Lambda_{\omega_0} \cap \Lambda_{\omega_1} = \Lambda_{[0,1]} \ar[ru]^-{i_0} \ar[rd]_-{i_1}\\
     & \Lambda_{\omega_1} \ar[r]_-{R_{\omega_1}} & \Lambda^{\mathcal K, \Gamma_{\omega_1}.} 
    }\]
Accordingly, we have variant versions of quantum homologies
\[ \xymatrix{
     & \widetilde{\QH_0} \ar[rr]^{\otimes_{\Lambda_{\omega_0}} \Lambda^{\mathcal K, \Gamma_{\omega_0}}} & &\QH_0 \\
     \QH_{[0,1]} \ar[ru]^-{\otimes_{\Lambda_{[0,1]}}\Lambda_{\omega_0}} \ar[rd]_-{\otimes_{\Lambda_{[0,1]}}\Lambda_{\omega_1}}\\
     & \widetilde{\QH_1} \ar[rr]^{\otimes_{\Lambda_{\omega_1}} \Lambda^{\mathcal K, \Gamma_{\omega_1}}} &&\QH_1.    }\]
Here the notation is defined as follows:
\begin{itemize}
\item{} $\QH_{[0,1]} = H_*(M; \mathcal K) \otimes_{\mathcal K} \Lambda_{[0,1]}$.
\item{} $\widetilde{\QH_i} = \QH_{[0,1]} \otimes_{\Lambda_{[0,1]}} \Lambda_{\omega_i}$ for $i = 0, 1$.
\item{} $\QH_i =\widetilde{\QH_i}  \otimes_{\Lambda_{\omega_i}} \Lambda^{\mathcal K,\Gamma_{\omega_i}}$ for $i =0, 1$. 
\end{itemize}

To end this subsection, we consider a similar short exact sequence as (\ref{ses}) but with multi-valuation from $\omega_0$ and $\omega_1$. Explicitly, consider the following short exact sequence, 
\[ 0 \to \ker[\omega_0] \cap \ker[\omega_1] \to H_2^S(M) \xrightarrow{[\omega_0] \times [\omega_1]} \Gamma_{\omega_0} \times \Gamma_{\omega_1} \to 0.\]
This allows us to identify any $x = \sum_{A \in H_2^S(M)} a_A T^A \in \Lambda_{[0,1]}$ with 
\begin{equation} \label{geo-elm}
 x = \sum_{(g_0, g_1) \in \Gamma_{\omega_0} \times \Gamma_{\omega_1}} a_{(g_0, g_1)} T^{(g_0, g_1)} \,\,\,\,\mbox{where}\,\,\, a_{(g_0, g_1)} \in \K[\ker[\omega_0] \cap \ker[\omega_1]]. 
 \end{equation}
In other words, each $x \in \Lambda_{[0,1]}$ can be identified with a set of points of $\R^2$ in coordinates $g_0, g_1$. Moreover, by the multi-finiteness condition, this set is discrete in $\R^2$. 

\subsection{Floer chain complex with multi-finiteness condition}\label{subsec-st-htp}
Given a Hamiltonian system $(M,\omega, H)$, set $\omega_0 = \omega$ and its reduced perturbation $\omega_1$. Similarly to the construction of a Floer chain complex, in this subsection we will construct a variant Floer chain complex over $\Lambda_{[0,1]}$. Define a graded finitely generated free $\Lambda_{[0,1]}$-module $(\CF_{[0,1]})_*$ as 
\begin{equation} \label{dfn-multi-cpx}
(\CF_{[0,1]})_k : = \bigoplus_{i=1}^n \Lambda_{[0,1]} \left<[x_i, w_i]\right> \simeq \bigoplus_{i=1}^n \Lambda_{[0,1]}
\end{equation}
where $n$ is the number of contractible Hamiltonian 1-periodic orbits $x_i \in \P(\omega_0, H)$ with CZ-index equal to $k$ and $w_i$ is a disk spanning $x_i$. Note that (\ref{dfn-multi-cpx}) is well-defined due to Proposition \ref{no-other-orbits}. Moreover, Lemma \ref{int-prop} implies
\begin{equation} \label{multi-t} 
(\CF_{[0,1]})_k = (\widetilde{\CF}_0)_k \cap (\widetilde{\CF}_1)_k =\,\, \bigcap_{t \in [0,1]} ({\widetilde{\CF}_t})_k
\end{equation}
for each degree $k \in \Z$. Here ${\widetilde{\CF}_t}$ is a free $\Lambda_{\omega_t}$-module for any $t \in [0,1]$. 

In order to form a chain complex, we need to choose a boundary operator on $(\CF_{[0,1]})_*$. For any $s \in [0,1]$, let $\omega_s = (1-s)\omega_0 + s \omega_1$, and simply denote by $\partial_s$ the standard Floer boundary operator of $(\CF_*(M, J, H, \omega_s), \partial_{J, H, \omega_s})$. The following Proposition shows that we have a family of boundary  operators for $(\CF_{[0,1]})_*$. 

\begin{prop} \label{well-dfn of bo} For any $s \in [0,1]$, $\partial_s$ is well-defined on $(\CF_{[0,1]})_*$ and satisfies $\partial_s^2 =0$. \end{prop}

\begin{proof} Since $\partial_s$ is already a well-defined boundary operator for the Floer chain complex $(\CF_*(M, J, H, \omega_s), \partial_s)$, the algebraic relation $\partial_s^2$ holds. In order to show that $\partial_s$ is well-defined on $(\CF_{[0,1]})_*$, we need to show that for any $t \in [0,1]$, the output of $\partial_s$ satisfies the finiteness condition in terms of $\omega_t$. Suppose that for a {\it basis element} $[x,w]$, 
\begin{equation} \label{boundary-s}
 \partial_s([x,w]) = \sum_{[y,v] \in \mbox{\tiny{basis}}} \sum_{A \in H_2^S(M)} n_A T^A[y,v]. 
 \end{equation}
By definition, there exists a Floer trajectory $u$, with respect to symplectic structure $\omega_s$, connecting $[x,w]$ and $[y, v \#A]$. We claim that there exists some constant $C'_{s,t}<1$ such that 
\begin{equation} \label{ene-bd}
 \mathcal A_{\omega_t}(T^A [y,v]) - \mathcal A_{\omega_t}([x,w]) \leq -(1-C'_{s,t})E(u). 
\end{equation}
In fact, first we have $\mathcal A_{\omega_s}(T^A[y,v]) - \mathcal A_{\omega_s}([x,w]) = - E(u)$. Here $E(u)$ is the energy of Floer trajectory under symplectic structure $\omega_s$. By the definition of the symplectic action functional with respect to $\omega_t$, 
\begin{align*}
\mathcal A_{\omega_t} ([x,w]) & = - \int_{D^2} w^*\omega_t + \int_0^1 H(x(t),t)dt \\
& = - \int_{D^2} w^*\omega_s + \int_0^1 H(x(t),t) dt + \int_{D^2} w^*(\omega_s - \omega_t) \\
& = \mathcal A_{\omega_s}([x,w]) + (s-t) \int_{D^2} w^*\alpha
\end{align*}
where $\alpha = \omega_1 - \omega_0$. Similarly, 
\[ \mathcal A_{\omega_t}(T^A[y,v]) = \mathcal A_{\omega_s}(T^A[y,v]) + (s-t) \int_{D^2} (v\#A)^*\alpha. \]
Therefore, one gets 
\[  \mathcal A_{\omega_t}(T^A [y,v]) - \mathcal A_{\omega_t}([x,w]) = - E(u) +  (s-t) \left(\int_{D^2} (v\#A)^*\alpha - \int_{D^2} w^*\alpha \right). \]
By our equivalence relation, $\int_{D^2} (v\#A)^*\alpha = \int_{D^2} (w \# u)^* \alpha$. Hence, by (\ref{a-energy}), there exists a constant $C$ such that 
\[  \int_{D^2} (v\#A)^*\alpha - \int_{D^2} w^*\alpha  = \int_{D^2} (w\#u)^* \alpha - \int_{D^2} w^*\alpha  = \int_{\R \times S^1} u^* \alpha \leq C|\alpha|E(u). \]
Meanwhile, as long as $\omega_1$ is {\it a priori} chosen sufficiently close to $\omega_0$, one gets $C|\alpha|<1$. Therefore by setting $C'_{s,t} = |s-t|C|\alpha|$, we get the claim (\ref{ene-bd}).

Next, (\ref{ene-bd}) can be rewritten as 
\[ -\int_{S^2} A^* \omega_t - N_2 \leq - (1- C'_{s,t}) E(u) \]
where $N_2 =\mathcal A_{\omega_t}([x,w])  -\mathcal A_{\omega_t}([y,v])$. This is a constant that is independent of the sphere class $A$. Since $1-C'_{s,t} >0$, 
\[ E(u) \leq \frac{1}{1-C'_{s,t}} \left(\int_{S^2}A^*\omega_t + N_2 \right). \]
If $\int_{S^2} A^* \omega_t < \lambda$ for $\lambda \in \R$, then 
\[ E(u) \leq \frac{\lambda + N_2}{1-C'_{s,t}} < \infty.\]
By the finiteness condition of $\omega_s$, there are only finitely many such sphere classes $A$. Hence, (\ref{boundary-s}) also satisfies the finiteness condition of $\omega_t$.  \end{proof}

For each $((\CF_{[0,1]})_*, \partial_s)$, we can associate a filtration function $\ell_{\omega_s}$  using the symplectic action functional $\mathcal A_{\omega_s}$ together with the valuation on $\Lambda_{[0,1]}$ with respect to $\omega_s$. Explicitly, for any chain $c = \sum_i \lambda_i [\gamma_i, w_i] \in (\CF_{[0,1]})_*$ with $\lambda_i \in \Lambda_{[0,1]}$, 
\begin{equation} \label{fil-s}
\ell_{\omega_s} (c) = \max_{i} \{ \mathcal A_{H, \omega_s}([x_i, w_i])  - \nu_{\omega_s} (\lambda_i)\}
\end{equation}
where $\nu_{\omega_s}(\lambda_i)$ denotes the minimal exponent from $\lambda_i$ after they are evaluated by the symplectic structure $\omega_s$. For brevity, denote $\ell_{\omega_s}$ as $\ell_s$. Proposition \ref{well-dfn of bo} says that there exists a family of variant version of Floer chain complexes $\{((\CF_{[0,1]})_*, \partial_s, \ell_s)\}_{s \in [0,1]}$ which is parametrized by $[0,1]$. It is important to use the filtration function $\ell_s$ and the boundary operator $\partial_s$ together so that this boundary operator strictly decreases the filtrations. In general, we get a family of Floer chain complexes 
\[ \{(\CF_{\Delta(\omega)}, \partial_{\omega'}, \ell_{\omega'})\}_{\omega' \in \Delta(\omega)}. \]
In a recent paper \cite{Le15}, its Theorem 3.12 provides a similar construction of a family of Floer-style chain complexes. 

Next, we show that there exists an algebraic relation between any two slices of these variant Floer chain complexes. 

\begin{prop} \label{st-htp} For any $s, t \in [0,1]$, the two complexes $((\CF_{[0,1]})_*, \partial_s, \ell_s)$ and $((\CF_{[0,1]})_*, \partial_t, \ell_t)$ are chain homotopy equivalent. \end{prop}

\begin{remark} Different from Lemma \ref{quasi-equ}, we can not get any quantitative comparison conclusion in Proposition \ref{st-htp}. From the proof given below, one can see that this comes from the ``non-uniform'' estimation from Proposition \ref{ene-est-cor2}, i.e., the bound depends on the symplectic area of the disk spanning the Hamiltonian 1-periodic orbit at the asymptotic end $s= -\infty$. On the other hand, one can view the family $\{((\CF_{[0,1]})_*, \partial_s, \ell_s)\}_{s \in [0,1]}$ from a different perspective. Since each $s$-slice $((\CF_{[0,1]})_*, \partial_s, \ell_s)$ provides a persistence module (\cite{CZ09}), this family provides a 2-dimensional persistence module. Continuity questions studied in this paper might be transferred into a stability problem of the invariants constructed from a higher dimensional persistence module. 
\end{remark}

\begin{proof} [Proof of Proposition \ref{st-htp}] We need to find a quadruple $(\Phi_{s,t}, \Phi_{t,s}, K_s, K_t)$ such that $\Phi_{s,t}$ and $\Phi_{t,s}$ are chain maps between $(\CF_{[0,1]})_*, \partial_s, \ell_s)$ and $(\CF_{[0,1]})_*, \partial_t, \ell_t)$ and $K_s, K_t$ are homotopies. First, choose a homotopy from $\omega_s$ to $\omega_t$ as in (\ref{dfn-int-htp}) parametrized by $\tau$, and consider the Floer operator $\mathcal F^{\tau}$ defined in Definition \ref{def-FO}. For any basis element $[x,w]$, define 
\begin{equation} \label{chain-map}
\Phi_{s,t}([x,w]) = \sum_{\tiny{\begin{array}{cc} [y,v] \in \mbox{basis} \\ \mu_{CZ}([x,w]) = \mu_{CZ}([y,v \#A]) \end{array}}} \sum_{A \in H_2^S(M)} n_A T^A [y,v] 
\end{equation} 
where $n_A$ counts the number of admissible $\mathcal F^{\tau}$-trajectories connecting $[x,w]$ and $[y, v \#A]$. We know that $\Phi_{s,t}$ is a chain map by the standard gluing argument. In order to show $\Phi_{s,t}$ acts on $(\CF_{[0,1]})_*$, we need to check that the output of $\Phi_{s,t}$ satisfies the finiteness condition of $\omega_r$ for any $r \in [0,1]$. Without loss of generality, assume $s<t$. By Proposition \ref{ene-est-cor2}, there exists a constant $C_{s,t}$ such that, 
\[ \mathcal A_{\omega_t}(T^A[y,v]) - \mathcal A_{\omega_s} ([x,w]) \leq -(1-C_{s,t}|\alpha|) E(u) + (s-t) \int_{D^2} w^* \alpha.\]
Meanwhile, evaluate $T^A[y,v]$ by $\mathcal A_{\omega_r}$ and one gets the following relation, 
\[ \mathcal A_{\omega_t}(T^A[y,v])  = \mathcal A_{\omega_r}(T^A [y,v]) + (r-t) \int_{S^1 \times \R} u^*\alpha + (r-t) \int_{D^2} w^*\alpha\]
where $\alpha = \omega_1 - \omega_0$. By (\ref{a-energy}) there exists some constant $C$ such that 
{\small \begin{align*}
 \mathcal A_{\omega_r}(T^A[y,v]) - \mathcal A_{\omega_s} ([x,w]) & \leq -(1-C_{s,t}|\alpha|) E(u) + (t-r) \int_{S^1 \times \R} u^*\alpha + (s-r) \int_{D^2} w^*\alpha\\
 & \leq  -(1-C_{s,t}|\alpha|) E(u) + C|\alpha|\cdot (t-r) E(u) + (s-r) \int_{D^2} w^* \alpha\\
 & \leq -(1-C'_{r,s,t}) E(u) + (s-r) \int_{D^2} w^*\alpha
\end{align*}}
for some constant $C'_{r,s,t}= (C_{s,t} + C \cdot (t-r))|\alpha|<1$. Also $ \mathcal A_{\omega_r}(T^A[y,v]) - \mathcal A_{\omega_s} ([x,w])  = -\int_{S^2} A^* \omega_r - N_3$ where $N_3 = \mathcal A_{\omega_s}([x,w])  -\mathcal A_{\omega_r}([y,v])$, independent of the sphere class $A$. Therefore, 
\[ E(u) \leq \frac{1}{1- C'_{r,s,t}} \left( \int_{S^2} A^* \omega_r + N_3 + (s-r)\int_{D^2} w^* \alpha \right).\]
If $\int_{S^2} A^* \omega_r \leq \lambda$ for any given $\lambda \in \R$, then since $1- C'_{r,s,t} >0$ for any $r \in [0,1]$,
\[ E(u) \leq \frac{\lambda + N_3 + (s-r)\int_{D^2} w^* \alpha}{1- C'_{r,s,t}} < \infty. \]
By Gromov compactness theorem, there are only finitely many sphere classes $A$. Hence, the output of (\ref{chain-map}) also satisfies the finiteness condition of $\omega_r$ for any $r \in [0,1]$.  Symmetrically, we can define $\Phi_{t,s}$ and prove it is a well-defined chain map. 

Second, $\Phi_{t,s} \circ \Phi_{s,t}$ is (Floer) homotopic to $\mathds{1}_s$ on $((\CF_{[0,1]})_*, \partial_s)$, by the standard Floer theory. Choose a homotopy between a symmetric homotopy of $\omega_s$ passing through $\omega_t$ and the identity homotopy $\omega_s$. We can construct a map $K_s$ as in Lemma \ref{quasi-equ} as follows,  
\begin{equation} \label{s-htp}
 K_s ([x,w]) = \sum_{\tiny{\begin{array}{cc} [y,v] \in \mbox{basis} \\ \mu_{CZ}([x,w]) +1= \mu_{CZ}([y,v\#A]) \end{array}}} \sum_{A \in H_2^S(M)} n_A T^A [y,v] 
 \end{equation}
where $n_A$ counts the number of pairs $(u, \lambda)$ where $u$ is an admissible $\mathcal F^{\tau, \lambda}_{\rm sym}$-trajectory connecting $[x,w]$ and $[y, v \# A]$. Again, in order to show $K_s$ acts on $(\CF_{[0,1]})_*$, we need to check that the output of $K_s$ satisfies the finiteness condition of $\omega_r$ for any $r \in [0,1]$. By Proposition \ref{ene-est-cor3} (applied to $\omega_0 = \omega_s$ and $\omega_1 = \omega_t$), 
\[ A_{\omega_s}(T^A [y,v]) - \mathcal A_{\omega_s}([x,w]) \leq -(1-C_{s,t}|\alpha|)E(u).\]
Evaluate $T^A[y,v]$ by $\mathcal A_{\omega_r}$, and one gets the following relation, 
\[ \mathcal A_{\omega_s}(T^A[y,v])  = \mathcal A_{\omega_r}(T^A [y,v]) + (r-s) \int_{S^1 \times \R} u^*\alpha + (r-s) \int_{D^2} w^*\alpha.\]
Hence, by a similar computation as above, 
\[  \mathcal A_{\omega_r}(T^A[y,v]) - \mathcal A_{\omega_s} ([x,w]) \leq -(1-C_{s,r}|\alpha|) E(u) + (s-r) \int_{D^2} w^*\alpha. \]
Then the same argument as above implies the finiteness condition of $\omega_r$. Thus we get the conclusion. \end{proof}

\subsection{Revised Floer homology} Denote by ${\rm CM}_{[0,1]}$ the coefficient extension of Morse chain complex ${\rm CM}_*(M; \mathbb Z)$ over ring $\Lambda_{[0,1]}$. In the following diagram, let us summarize the relations between various chain complexes that we have encountered so far. 
\begin{equation} \label{C}
\xymatrix{
  & (\CF_{[0,1]}, \partial_0) \ar[r]^-{\iota_0} \ar[d]^-{\Phi_{0,t}} & ({\widetilde {\CF}}_0, \partial_0) \ar[r]^-{R_0} \ar@{-->}[d]^-{\tiny{\mbox{no well-defined chain map}}} & (\CF_*(M, H, J, \omega_0), \partial_0) \\
({\rm CM}_{[0,1]}, \partial_{\rm Morse}) \ar[r]^-{PSS_t} \ar[ur]^-{PSS_0} \ar[rd] & (\CF_{[0,1]}, \partial_t)  \ar[d] \ar[r]^-{\iota_t} & ({\widetilde {\CF}}_t, \partial_t) \ar@{-->}[d]  \\
  & (\CF_{[0,1]}, \partial_1) \ar[r]^-{\iota_1} \ar@{~>}[d] & ({\widetilde {\CF}}_1, \partial_1) \ar[r]^-{R_1} & (\CF_*(M, H, J, \omega_1), \partial_1) \\
  & \mbox{$\Lambda_{[0,1]}$-module}}
  \end{equation}
where ${\widetilde {\CF}}_t = \CF_{[0,1]} \otimes_{\Lambda_{[0,1]}} \Lambda_{\omega_t}$ is a free $\Lambda_{\omega_t}$-module for any $t \in [0,1]$. Take the homology of each chain complex, and one gets the following picture, 
\begin{equation} \label{H}
\xymatrix{
  & \HF_{[0,1], 0} \ar[r]^-{(\iota_0)_*} \ar[d]^-{(\Phi_{0,t})_*} & {\widetilde {\HF}}_0 \ar[r]^-{(R_0)_*} \ar@{-->}[d]^-{\tiny{\mbox{no well-defined map}}} & \HF_{\omega_0} \\
\QH_{[0,1]} \ar[r]^-{(PSS_t)_*} \ar[ur]^-{(PSS_0)_*} \ar[rd] & \HF_{[0,1],t} \ar[d] \ar[r]^-{(\iota_t)_*} & {\widetilde {\HF}}_t\ar@{-->}[d]  \\
  & \HF_{[0,1], 1} \ar[r]^-{(\iota_1)_*} \ar@{~>}[d] & {\widetilde {\HF}}_1 \ar[r]^-{(R_1)_*} & {\HF}_{\omega_1} \\
  & \mbox{$\Lambda_{[0,1]}$-module}}
 \end{equation}
where $\QH_{[0,1]} = H_*(M; \K) \otimes \Lambda_{[0,1]}$. Moreover, $\HF_{[0,1], t}$ is the homology of Floer chain complex $((\CF_{[0,1]})_*, \partial_t)$, ${\widetilde {\HF}}_t$ is the homology of Floer chain complex $(({\widetilde {\CF}}_t)_*, \partial_t)$ for any $t \in [0,1]$ and $\HF_{\omega_i}$ is the homology of Floer chain complex $(\CF_*(M, H, J, \omega_i), \partial_i)$ for $i =0, 1$.\\

By Proposition \ref{st-htp}, $\HF_{[0,1], t}$ are all isomorphic to each other. Moreover, one can show that for any $t \in [0,1]$, 
\[ (\Phi_{0,t})_* \circ (PSS_0)_* = (PSS_t)_*. \]
A natural question is how these Floer homologies change when we extend the coefficients in each step. First, Universal Coefficient Theorem (Corollary 7.56 (ii) and Theorem 7.15 in \cite{Rot09}) says that, for each degree $k \in \Z$, we have the following splitting,
\begin{equation} \label{uct-splitting} 
H_k(\widetilde{\CF}_i; \Lambda_{\omega_i}) \simeq H_k(\CF_{[0,1]}; \Lambda_{[0,1]}) \otimes_{\Lambda_{[0,1]}} \Lambda_{\omega_i} \oplus {\rm Tor}^{\Lambda_{[0,1]}}(H_{k-1}(\CF_{[0,1]}), \Lambda_{\omega_i}) 
\end{equation}
where ${\rm Tor}^{\Lambda_{[0,1]}} (H_{k-1}(\CF_{[0,1]}), \Lambda_{\omega_i})$ is a torsion module over $\Lambda_{[0,1]}$. It is not easy to see the algebraic relation between $\Lambda_{[0,1]}$ and $\Lambda_{\omega_i}$ if we try to apply some well-known fact such as that a module over a PID is flat if and only if it is torsion-free (by Lemma \ref{int-domain}, we only know $\Lambda_{[0,1]}$ is a domain). Fortunately, we still have the following property claiming that the torsion part vanishes, due to the existence of PSS-maps who transfer our discussion back to the Morse homology. 

\begin{prop} \label{no-torsion} For any $k \in \Z$, 
\[ H_k(\widetilde{\CF}_i; \Lambda_{\omega_i}) \simeq H_k(\CF_{[0,1]}; \Lambda_{[0,1]}) \otimes_{\Lambda_{[0,1]}} \Lambda_{\omega_i}\]
for $i =0, 1$, or simply $\widetilde{\HF}_i \simeq \HF_{[0,1]} \otimes_{\Lambda_{[0,1]}} \Lambda_{\omega_i}$. In particular 
\[ {\r}_{\Lambda_{\omega_0}} H_k(\widetilde{\CF}_0; \Lambda_{\omega_0}) = {\r}_{\Lambda_{\omega_1}} H_k(\widetilde{\CF}_1; \Lambda_{\omega_1}).\]
 \end{prop}

\begin{proof} Recall that ${\rm CM}_{[0,1]}  = {\rm CM}(M; \Z) \otimes_{\Z} \Lambda_{[0,1]}$. By Universal Coefficient Theorem, 
\[ H_k({\rm CM}_{[0,1]}; \Lambda_{[0,1]}) \simeq H_k({\rm CM}) \otimes_{\Z} \Lambda_{[0,1]} \oplus {\rm Tor^{\Z}}(H_{k-1}({\rm CM}_{[0,1]}), \Lambda_{[0,1]}) \] 
where $H_*({\rm CM}) := H_*({\rm CM}; \Z)$. By Lemma \ref{int-domain}, $\Lambda_{[0,1]}$ is an integral domain, so it is torsion-free as a $\Z$-module, which implies the flatness since $\Z$ is a PID. Therefore, Tor functor vanishes, that is, 
\begin{equation} \label{01} 
H_*({\rm CM}) \otimes_{\Z} \Lambda_{[0,1]} \simeq H_*({\rm CM}_{[0,1]}; \Lambda_{[0,1]}).
\end{equation}
By the same argument, 
\begin{equation} \label{03}
H_*({\rm CM}) \otimes_{\Z} \Lambda_{\omega_i} \simeq H_*({\rm CM} \otimes_{\Z} \Lambda_{\omega_i}; \Lambda_{\omega_i}).
\end{equation}
Together, we get the following relations, 
\[ H_*({\rm CM})\otimes_{\Z} \Lambda_{\omega_i} = \left(H_*({\rm CM}) \otimes_{\Z} \Lambda_{[0,1]}\right) \otimes_{\Lambda_{[0,1]}} \Lambda_{\omega_i} \simeq  H_*({\rm CM}_{[0,1]}; \Lambda_{[0,1]})\otimes_{\Lambda_{[0,1]}} \Lambda_{\omega_i}.  \]

On the other hand, similarly to (\ref{uct-splitting}), we have 
\[ H_*({{\rm CM}_i}; \Lambda_{\omega_i}) \simeq H_*({\rm CM}_{[0,1]}; \Lambda_{[0,1]}) \otimes_{\Lambda_{[0,1]}} \Lambda_{\omega_i} \oplus {\rm Tor}^{\Lambda_{[0,1]}} (H_{*-1}({\rm CM}_{[0,1]}), \Lambda_{\omega_i}) \]
where ${{\rm CM}_i} = {\rm CM}_{[0,1]} \otimes_{\Lambda_{[0,1]}} \Lambda_{\omega_i}$. Consider the following commutative diagram
\[
{\small \xymatrixcolsep{1pc} \xymatrix{ 
(H_*({\rm CM}) \otimes_{\Z} \Lambda_{[0,1]}) \otimes_{\Lambda_{[0,1]}}\Lambda_{\omega_i} \ar[d]^j \ar@/_5pc/[dd]_t\\
H_*({\rm CM}) \otimes_{\Z} \Lambda_{\omega_i} \ar[r]^-{q} \ar[d]^s & H_*({\rm CM} \otimes_{\Z} \Lambda_{\omega_i}; \Lambda_{\omega_i}) \ar[d]^f \\
H_*({\rm CM}_{[0,1]}; \Lambda_{[0,1]}) \otimes_{\Lambda_{[0,1]}} \Lambda_{\omega_i} \ar[r]^-i \ar[d]^g & H_*({\rm CM}_i; \Lambda_{\omega_i}) \ar[r] \ar[d]^h& {\rm Tor}^{\Lambda_{[0,1]}}(H_{*-1}({\rm CM}_{[0,1]}), \Lambda_{\omega_i}) \ar[d]\\
H_*(\CF_{[0,1]}; \Lambda_{[0,1]}) \otimes_{\Lambda_{[0,1]}} \Lambda_{\omega_i} \ar[r]^-p & H_*(\CF_i; \Lambda_{\omega_i}) \ar[r] & {\rm Tor}^{\Lambda_{[0,1]}}(H_{*-1}(\CF_{[0,1]}), \Lambda_{\omega_i}).}}
\]
In this diagram,
\begin{itemize}
\item $f$ is an identity map because ${{\rm CM}_i} = {\rm CM} \otimes_{\Z} \Lambda_{\omega_i}$; 
\item $q$ is an isomorphism because of (\ref{03}); 
\item $g$ and $h$ are PSS-maps (see \cite{PSS96}), so isomorphisms; 
\item $j$ is an identity map due to the extension of coefficients;
\item  $t$ is an isomorphism because of (\ref{01}). 
\end{itemize}
Therefore, $t$ is an isomorphism, which implies that $s$ is an isomorphism. This implies $i$ is an isomorphism, and then $p$ is an isomorphism. \end{proof} 

Finally, since $\Lambda^{\mathcal K, \Gamma_{\omega_i}}$ is a field, any torsion over $\Lambda^{\mathcal K, \Gamma_{\omega_i}}$ always vanishes. Hence, one gets the following result. 

\begin{cor} \label{no-torsion2} For any $k \in \Z$, 
\[ H_k(\CF(M, J, H, \omega_i); \Lambda^{\mathcal K, \Gamma_{\omega_i}}) \simeq H_k(\widetilde{\CF}_i; \Lambda_{\omega_i}) \otimes_{\Lambda_{\omega_i}} \Lambda^{\mathcal K, \Gamma_{\omega_i}}\]
for $i =0, 1$, or simply $\HF_{\omega_i} = \widetilde{\HF}_{[0,1]} \otimes_{\Lambda_{[0,1]}} \Lambda_{\omega_i}$. In particular,
\[ {\r}_{\Lambda^{\K, \Gamma_{\omega_0}}} (\HF_{\omega_0})_k = {\r}_{\Lambda^{\K, \Gamma_{\omega_1}}} (\HF_{\omega_1})_k. \]
\end{cor}

\section{Variant spectral invariants; proof of Theorem \ref{upper-semi}} \label{sec-si} 

Fix any class $a \in \QH_{[0,1]}$. Using $(PSS_t)_*$, one gets a class $(PSS_t)_*(a)$ in $\HF_{[0,1], t}$. Recall that for each $t \in [0,1]$, $(\CF_{[0,1]}, \partial_t, \ell_t)$ is a filtered complex with respect to the symplectic structure $\omega_t$. 

\begin{dfn} \label{t-si} Given a Hamiltonian system $(M, \omega, H)$, set $\omega_0 = \omega$, and $\omega_1$ is a reduced perturbation of $\omega_0$. For $a \in \QH_{[0,1]}$, we call the following value {\it $t$-spectral invariant} associated to $a$,
\[ \rho_{t}(a,H) = \inf\{\ell_t(\alpha_t)\,| \, [\alpha_t] = (PSS_t)_*(a) \}\]
where $\alpha_t \in (\CF_{[0,1]}, \partial_t)$. 
\end{dfn}

Recall that for each $t \in [0,1]$, there is a well-defined spectral invariant $\tilde{\rho}_t(a, H)$ over the coefficient ring $\Lambda_{\omega_t}$ for any $a \in \QH_{[0,1]}$. This is the abstract spectral invariant defined in \cite{Ush08}. Meanwhile, $\rho(a, H; \omega_t)$ denotes the standard spectral invariant over the Novikov {\bf field} $\Lambda^{\mathcal K, \Gamma_{\omega}}$. It is readily to see that $\tilde{\rho}_t(a, H) = {\rho}(H, a; \omega_t)$ for any $t \in [0,1]$ and any $a \in \QH_{[0,1]}$. On the other hand, we have the following important property. 

\begin{lemma} \label{prop-si} Let $t \in [0,1]$ and a non-degenerate Hamiltonian $H \in C^{\infty}(\R/\Z \times M)$.
\begin{itemize}
\item[(1)] (finiteness) For any nonzero $a \in \QH_{[0,1]}$, $\rho_t(a, H) > - \infty$. 
\item[(2)] (realization) For any nonzero $a \in \QH_{[0,1]}$, there exists an $\alpha_t \in \CF_{[0,1]}$ such that 
\[ \rho_t(a, H) = \ell_t(\alpha_t).\]
\item[(3)] (extension) For any $a \in \QH_{[0,1]}$, $\rho_t(a, H) = \tilde\rho_t(a,H)$. 
\end{itemize}
\end{lemma}

Note that Lemma \ref{prop-si} reduces the comparison between $\rho(a, H; \omega_0)$ and $\rho(a, H; \omega_1)$ to the comparison between $\rho_0(a, H)$ and $\rho_1(a, H)$, which has the advantage that they are over the same coefficient ring $\Lambda_{[0,1]}$. 

\begin{proof} $(1)$: Regard $a = a\otimes \mathds{1}$ as an element in $\widetilde{\QH}_t$ over the coefficient ring $\Lambda_{\omega_t}$, still non-zero. By Theorem 1.3 in \cite{Ush08}, we know $\tilde{\rho}_t(a, H) > -\infty$. On the other hand, by definition, for any $\ep>0$, there exists some $\alpha_t \in (\CF_{[0,1]}, \partial_t)$ representing $[\alpha_t] = (PSS_t)_*(a)$ such that 
\[ \ell_t(\alpha_t) \leq \rho_t(a, H) + \ep.\]
Then, in $\widetilde{\CF}_t$ over the coefficient ring $\Lambda_{\omega_t}$, $\alpha_t (= \alpha_t \otimes \mathds{1})$ also represents $(PSS_t)_*(a)$. By definition, $\tilde{\rho}_t(a, H) \leq \ell_t(\alpha_t)  \leq \rho_t(a, H)+ \ep$. So $\rho_t(a, H) > -\infty$. Therefore, we get the conclusion $(1)$. 

$(2)$ and $(3)$: The same argument works for $t$-spectral invariants for any $t \in [0,1]$, so we only prove the case for $t = 0$. Since $a$ is a non-zero element in $\QH_{[0,1]}$, there exists some chain $\alpha \in \CF_{[0,1]}$ such that $(PSS_0)_*(a) = [\alpha]$. On the other hand, viewing $a$ as an element in $\widetilde{\QH}_0$ over the coefficient ring $\Lambda_{\omega_0}$, Theorem 1.4 in \cite{Ush08} says that there exists an optimal boundary $\partial_0 \tilde y$ in the sense that
\[ \tilde{\rho}_0(a, H) = \ell_0(\alpha - \partial_0 \tilde y)\]
where $\tilde{y} \in \widetilde{\CF}_0$ (but not necessarily in $\CF_{[0,1]}$). Decompose $y$ into the following two parts, 
\begin{equation} \label{decomp}
\tilde y = y_f + y_* 
\end{equation}
where $\ell_0(y_*) \leq \tilde{\rho}_0(a, H)$. Then by the finiteness condition, the sub-chain $y_f$ contains only finitely many terms. Also,
\[ \ell_0(\alpha - \partial_0 y_f) = \ell_0(\alpha - \partial_0 \tilde y + \partial_0 y_*) = \ell_0 (\alpha - \partial_0 \tilde y) \]
where the final equality comes from the relation $\ell_0(\partial_0 y_*) < \tilde{\rho}_0(a, H) = \ell_0(\alpha - \partial_0 \tilde y)$. In other words, $\partial_0 y_f$ is also an optimal boundary for spectral invariant $\tilde{\rho}_0(a, H)$. Since $y_f$ has only finitely many terms, certainly $y_f \in \CF_{[0,1]}$. Proposition \ref{well-dfn of bo} implies $\partial_0 y_f$ is also in $\CF_{[0,1]}$, and then $\alpha - \partial_0 y_f$ is an element in $\CF_{[0,1]}$. Therefore,
\begin{equation} \label{sp-ine1}
\tilde{\rho}_0(a, H) = \ell_0(\alpha- \partial_0 y_f) \geq \rho_0(a, H).
\end{equation}
On the other hand, for any $\ep>0$, there exists some $\alpha' \in \CF_{[0,1]}$ such that $[\alpha'] = (PSS_0)_*(a)$ and $\ell_0(\alpha') \leq \rho_0(a, H) + \ep$. Viewing $\alpha'$ as an element in $\widetilde{\CF}_0$ and $a$ as an element in $\widetilde{\QH}_0$, $\alpha'$ also represents $a$ and hence $\tilde{\rho}_0(a, H) \leq \ell_0(\alpha')$. Therefore, 
\begin{equation} \label{sp-ine2}
\tilde{\rho}_0(a, H) \leq \ell_0(\alpha') \leq \rho_0(a, H) + \ep.
\end{equation}
Since this is true for any $\ep>0$, we get $\tilde{\rho}_0(a,H) \leq \rho_0(a,H)$, which together with (\ref{sp-ine1}) finishes the proof.  \end{proof}
 
Before giving the proof of Theorem \ref{upper-semi}, we will prove the following lemma on the continuity of filtration functions on a fixed chain in $\CF_{[0,1]}$. 

\begin{lemma} \label{pc-fil} For any fixed chain $c \in \CF_{[0,1]}$, the function $t \mapsto \ell_t(c)$ is continuous on $[0,1]$. \end{lemma}

\begin{proof} We will only prove the continuity at $t =0$. For any other $t \in (0,1]$, the proof is exactly the same. First, suppose that $\CF_{[0,1]}$ is a free $\Lambda_{[0,1]}$-module of rank $n$, we can identify $c$ as an $n$-tuple $\vec{x}$ in $\Lambda^n_{[0,1]}$. Moreover, by (\ref{geo-elm}), we can write 
\begin{equation} \label{x}
 \vec{x} = \sum_{(g_0, g_1)} \vec{a}_{(g_0, g_1)} T^{(g_0, g_1)} \,\,\,\,\mbox{where}\,\,\, \vec{a}_{(g_0, g_1)} \in (\K[\ker[\omega_0] \cap \ker[\omega_1]])^n.
 \end{equation}
Moreover, $\vec{x}$ can be further identified with a set of points $(g_0, g_1)$ on the $g_0g_1$-plane. By the finiteness condition of both $\omega_0$ and $\omega_1$,  up to a uniform shift on both indices, we can assume that all the points are lying in the first quadrant. 

Second, by definition, $\omega_t = (1-t) \omega_0 + t \omega_1$. Let $t := \frac{1}{1+ \lambda}$ for some non-negative $\lambda$ and define 
\[ \omega_{\lambda} = \lambda \omega_0 + \omega_1 (= (1+\lambda) \omega_t).\]
Note that over a set of homological spheres, $\omega_t$ obtains its minimal value if and only if $\omega_{\lambda}$ obtains its minimal value. One way of viewing $\ell_t$ is via the perturbation of valuation function $\bar{\nu}_t$. Specifically, for any $c \in \CF_{[0,1]}$, 
\[ \bar{\nu}_t(c) = \min\left\{\int_{S^2} A^*\omega_t \,\bigg| \, \mbox{$A$ is an exponent of $c$}\right\},\]
and then
\[ \ell_t(c) = -\bar{\nu}_t(c) + p_t(c), \]
where $p_t(c)$ comes from the Hamiltonian actions on Hamiltonian 1-periodic orbits, as well as the symplectic areas of the (fixed) cappings of basis elements. As $p_t(c)$ eventually goes to $p_0(c)$ when $t \to 0$, it suffices to focus on $\bar{\nu}_t(c)$ when studying the continuity of $\ell_t(c)$. Actually, we will focus on $(1+\lambda) \bar{\nu}_t(c)$, that is 
\begin{align*}
 \bar{\nu}_{\lambda}(c) &:= \min\left\{\int_{S^2} A^* \omega_{\lambda} \, \bigg| \, \mbox{$A$ is an exponent of $c$}\right\} \\
 &= \min\{\lambda g_0 + g_1\,| \,  \mbox{$(g_0, g_1)$ is an exponent of $\vec{x}$}\}.
\end{align*}
Once rephrased in this way, it suggests a geometric way to view the value $\bar{\nu}_{\lambda}(c)$: for any $\lambda \geq 0$ and for any point $(g_0, g_1)$, draw a line passing through $(g_0, g_1)$ with slope $- \lambda$, that is 
\begin{equation} \label{line}
y = - \lambda (x- g_0) + g_1.
\end{equation}
Then the minimal $y$-intercept is just the value $\bar{\nu}_{\lambda}(c)$. The nontrivial part is that the optimal point $(g_0, g_1)$ who attains the minimal $y$-intercept might change along the change of $\lambda$ (equivalently the change of $t$). However, we claim that when $\lambda >>0$, there exists a point $(g_0^*, g_1^*)$ who serves as the optimal choice for all sufficiently large $\lambda$. 

The key observation is that for any point $P = (g_0, g_1)$ attaining the value $\bar{\nu}_{\lambda}(c)$ for some $\lambda$, it fails to attain the value $\bar{\nu}_{\eta}(c)$ for any $\eta > \lambda$ if there exists another point $Q$ in the region enclosed by $y$-axis, the line (\ref{line}) passing through $(g_0, g_1)$ with slope $-\lambda$ and the line (\ref{line}) passing through $(g_0, g_1)$ with slope $-\eta$. When $\lambda \to \infty$, the width of this closed region goes to zero. Hence, by the discreteness of our points, the choice of an optimal point will be eventually stable. Hence, 
\[ \bar{\nu}_t(c) = \frac{\lambda g_0^* + g_1^*}{1+ \lambda} \xrightarrow{\lambda \to \infty} g_0^*  =\bar{\nu}_0(c). \]
The limit $\lambda \to \infty$ is equivalent to $t \to 0$. Thus we get the conclusion. 
\end{proof} 

The following proposition is the key step towards the proof of Theorem \ref{upper-semi}. A similar result of this type, but from a different set-up, is Proposition 8.4 in \cite{Oh09}.  Recall that in (\ref{chain-map}), we have defined a chain map over $\Lambda_{[0,1]}$, $\Phi_{0,t}: ((\CF_{[0,1]})_*, \partial_0, \ell_0) \to ((\CF_{[0,1]})_*, \partial_t, \ell_t)$ for any $t\in [0,1]$.

\begin{prop} \label{cont-cont} For any chain $c \in (\CF_{[0,1]}, \partial_0, \ell_0)$, the function $t \to \ell_t(\Phi_{0,t}(c))$ is upper semicontinuous at $t=0$.\end{prop}

\begin{proof} We will prove it by contrapositive. Suppose that there exists a constant $\ep_0>0$ and a sequence $t_n \to 0$ such that 
\begin{equation} \label{cp}
\ell_{t_n}(\Phi_{0, t_n}(c)) - \ell_0(c) \geq \ep_0.
\end{equation}
Since $\Phi_{0, t_n}(c) = c + \sigma_n$ where there exist non-trivial Floer trajectories between $c$ and $\sigma_n$, by the triangle inequality of $\ell_{t_n}$, 
\[ \ell_{t_n}(\Phi_{0, t_n}(c)) = \ell_{t_n}(c + \sigma_n) \leq \max\{\ell_{t_n}(c), \ell_{t_n}(\sigma_n)\}.\]
Then (\ref{cp}) implies that 
\begin{equation} \label{cp2}
\max\{\ell_{t_n}(c) - \ell_0(c), \ell_{t_n}(\sigma_n) - \ell_0(c)\} \geq \ep_0 >0. 
\end{equation} 
Lemma \ref{pc-fil} implies that the first term in (\ref{cp2}) will be smaller than $\ep_0$ when $n$ is sufficiently large. Therefore, (\ref{cp2}) is possible only if $\ell_{t_n}(\sigma_n) - \ell_0(c) \geq \ep_0$. 

Since there are only finitely many basis elements, by passing to a subsequence, we can assume that for each $n$, $\ell_{t_n}(\sigma_n) = \mathcal A_{\omega_{t_n}}(T^{A_n}[y,v])$, with the same basis element $[y,v]$. By definition, there exists some sub-chain of $c$ with elements connecting with $T^{A_n}[y,v]$ by Floer trajectories. Again, since there are only finitely many basis generators, by passing to a subsequence, we can assume that, 
\[ T^{A_n}[y,v] \,\,\mbox{is connected with} \,\,T^{B_n}[x,w] \,\,\,\,\mbox{by a Floer trajectory}\]
for a basis element $[x,w]$ and sphere classes $B_n$. In particular, set $\{T^{B_n}[x,w]\}_{n=1}^{\infty}$ is a collection of elements from chain $c$. Now, we claim that, 
\[ \left|\int_{S^2} B_n^* \alpha\right| \to \infty \,\,\,\, \mbox{where $\alpha = \omega_1- \omega_0$}.\]
In fact, since $\ell_0(c) \geq \mathcal A_{\omega_0}(T^{B_n}[x,w])$, Proposition \ref{ene-est-cor2} with $s = 0$ and $t = t_n$ implies the following inequalities, 
\begin{equation} \label{non-small}
 \ell_{t_n}(x_n) - \ell_0(c) \leq \mathcal A_{\omega_{t_n}}(T^{A_n}[y,v]) - \mathcal A_{\omega_0}(T^{B_n}[x,w]) \leq -t_n\int_{D^2} (w\#B_n)^* \alpha.
 \end{equation}
Therefore, if $|\int_{S^2} B_n^* \alpha|$ is bounded, then when $t_n$ is close to $0$, this violates (\ref{cp2}).  

Then, in particular, $B_n$ is not equal to any fixed homotopy class when $n$ is sufficiently large. The finiteness condition of chain $c$ under the symplectic structure $\omega_0$ implies, 
 \begin{equation} \label{infinity}
 \int_{S^2} B_n^*\omega_0 \to \infty.
 \end{equation}

Returning to the continuation chain map, we know 
\[ \Phi_{0, t_n}(T^{B_n}[x,w]) = T^{B_n}[x,w] + T^{A_n} [y, v] + \ldots.\]
Meanwhile, by $\Lambda_{[0,1]}$-linearity of $\Phi_{0,t}$, for any $m \in \N$, 
\[ \Phi_{0, t_n}(T^{B_m} [x,w]) = T^{B_m} [x,w] + T^{A_n + B_m - B_n} [y,v] + \ldots. \]
Since $T^{B_m} [x,w]$ is a generator of chain $c$, $T^{A_n + B_m - B_n} [y,v]$ will be a generator of chain $\sigma_n$. However, since $T^{A_n} [y,v]$ attains the maximal filtration with respect to $\omega_{t_n}$, for any $m$, $\int_{S^2} (A_n + B_m - B_n)^* \omega_{t_n} \geq \int_{S^2} A_n^*\omega_{t_n}$. This implies that 
\begin{equation} \label{infinity2}
\int_{S^2}B_m^*\omega_{t_n} \geq \int_{S^2} B_n^*\omega_{t_n}. 
 \end{equation}
Rewrite
\[ \int_{S^2} B_m^*\omega_{t_n} = \int_{S^2} B_m^*\omega_0 + t_n \int_{S^2}B_m^*\alpha: = a_m + t_n b_m,\]
where $a_m = \int_{S^2}B_m^*\omega_0$ and $b_m = \int_{S^2} B_m^*\alpha$. Moreover, denote $c_n = \int_{S^2}B_n^*\omega_{t_n}(=a_n + t_n b_n)$. Then (\ref{infinity2}) says that 
\[ a_m + t_n b_m \geq c_n. \]
Switch the index $m$ and $n$, then we get 
\[ a_n + t_m b_n \geq c_m.\]
Since $b_m = \frac{c_m - a_m}{t_m}$ and $b_n = \frac{c_n - a_n}{t_n}$, the inequalities above are 
\begin{equation} \label{ineq}
a_m + \frac{t_n}{t_m} (c_m - a_m)\geq c_n \,\,\,\,\mbox{and} \,\,\,\, a_n + \frac{t_m}{t_n} (c_n - a_n)\geq c_m.
\end{equation}
Solve $c_m$ from the first inequality, and then the second inequality implies that 
\[ a_n + \frac{t_m}{t_n} (c_n - a_n)\geq a_m + \frac{t_m}{t_n} (c_n - a_m). \]
This is equivalent to the following relation, 
\[ \left( \frac{t_m}{t_n} - 1\right) (a_m - a_n) \geq 0.\]
By (\ref{infinity}), we know $a_m > a_n$ when $m >>n$, which is strictly positive due to (\ref{infinity}). This implies $t_m > t_n$ which is a contradiction since $t_n$ converge to $0$ (so when $m > > n$, $t_m < t_n$). \end{proof}

Now, we are ready to give the proof of Theorem \ref{upper-semi}. 

\begin{proof}  [Proof of Theorem \ref{upper-semi}] For simplicity, we will only prove the case for two symplectic structures $\omega_0$ and $\omega_1$, that is, for any $t \in [0,1]$ and $a \in \QH_{[0,1]}$, the map $t \mapsto \rho_t(a, H)$ is upper semicontinuous. It is easy to see how this can be generalized to the general case.

By realization property (2) in Proposition \ref{prop-si}, there exists some $c \in (\CF_{[0,1]}, \partial_0)$ such that $\ell_0(c) = \rho_0(a, H)$ where $[c] = (PSS_0)_*(a)$. On the other hand, $\Phi_{0,t}(c)$ represents 
\[ [\Phi_{0,t}(c)] = (\Phi_{0,t})_*[c] = (\Phi_{0,t})_* ((PSS_0)_*(a)) = (PSS_t)_*(a).\]
Then by definition, we know that $\rho_t(a, H) \leq \ell_t(\Phi_{0,t}(c))$. Therefore,  
\[ \rho_{t}(a, H) - \rho_0(a, H) \leq \ell_t(\Phi_{0,t}(c)) - \ell_0(c).\]
Upper semicontinuity from Proposition \ref{cont-cont} implies desired conclusion. \end{proof} 

\begin{remark} \label{lower-semi} There is an obvious question on the lower semicontinuity of $t$-spectral invariants. A trial of imitating the proof of Theorem 8.3 in \cite{Oh09} can be carried out but some details could not go through deeply due to the non-uniform upper bound from the energy estimation from Proposition \ref{ene-est}. \end{remark}

\section{Applications on quasi-isometric embedding and capacity} \label{sec-emb}

\subsection{Proof of Theorem \ref{embedding}}

\begin{dfn}\label{ap} Let us give the following two definitions. 
\begin{itemize} 
\item[(1)]  (Definition 1.3 in \cite{Ush12}) We say $M$ admits an {\it aperiodic symplectic structure} if there exists a symplectic structure $\omega$ such that $(M, \omega)$ admits an autonomous Hamiltonian function $H$, not everywhere locally constant, such that its Hamiltonian flow has no nonconstant periodic orbit. Such $H$ is called an {\it associated aperiodic Hamiltonian} of $\omega$.
\item[(2)] For a symplectic structure $\omega$, we call it {\it strongly-aperiodic-approximated} (saa-condition) if there exists a sequence of aperiodic structures $\omega_n$ such that (a) $\omega_n \to \omega$ under the norm $|\cdot|$ in Section \ref{simple}; (b) there exists a sequence of associated aperiodic Hamiltonians $H_n$ of $\omega_n$ that $C^{\infty}$-converges to a differentiable function $H$ on $M$.
\end{itemize}
\end{dfn}

It is well-known that the computation of spectral invariants is difficult in general. The following theorem will be helpful in the proof later. 

\begin{prop} \label{sp-computation} (Proposition 4.1 in \cite{Ush10}) Let $(M, \omega)$ be a symplectic manifold. If $H$ is an autonomous Hamiltonian function on $M$ such that its Hamiltonian flow has no nonconstant contractible periodic orbit with the period at most $1$, then 
\[ \rho([M], H; \omega) = - \min_{M} H. \]
\end{prop}
Notice that the condition in this proposition is weaker than the assumption of Theorem 1.1 in \cite{Ush13}. In \cite{Ost03}, for {\it any} symplectic manifold $(M, \omega)$, it is proved that Hofer diameter of $\widetilde{\Ham}(M, \omega)$ is infinite by using, roughly speaking, a sequence of bump functions on a displaceable subset. Theorem \ref{embedding} shows that under a certain condition (which covers a variety of symplectic manifolds, especially in 4-dimensions) the Hofer diameter of $\widetilde{\Ham}(M, \omega)$ goes to infinity in uncountably many linearly independent directions. The proof of Theorem \ref{embedding} takes its inspiration from the proof of Theorem 1.1 in \cite{Ush13}. 

\begin{proof} [Proof of Theorem \ref{embedding}] First, note that if $\omega$ is already aperiodic, then in particular $(M, \omega)$ admits an autonomous Hamiltonian $H$ such that its Hamiltonian flow has no nonconstant contractible periodic orbit. Then by the compactness of $M$ and Sard's theorem, there exists a non-trivial closed interval $[a,b]$ (assumed to be $[0,1]$) such that every $c \in [0,1]$ is a regular value of $H$. Now, take a function $g: \R \to [0,1]$ such that its support is in $(0,1)$, $\max g= 1$ and its only local minimum has value $0$. For each $\vec{v} = (v_1, v_2, ...) \in \R^{\infty}$, define $f_{\vec{v}}: \R \to \R$ by
\[ f_{\vec{v}}(s) = \sum_{i=1}^{\infty} v_i\cdot g\left(2^{i} (s - (1- 2^{1-i}))\right). \]
The embedding $\Phi: \R^{\infty} \to \widetilde{\Ham}(M, \omega)$ is constructed as $\Phi(\vec{v}) = [\phi^1_{f_{\vec{v}} \circ H}]$. Then for any non-zero $\vec{v} \in \R^{\infty}$, $X_{f_{\vec{v}}\circ H} = f'_{\vec{v}}(H) \cdot  X_H$, which implies $\Phi$ is a homomorphism. Therefore, 
\begin{align*}
 d_{H} (\Phi(\vec{v}), \Phi(\vec{w})) &= d_{H}([\phi^1_{f_{\vec{v}} \circ H}], [\phi^1_{f_{\vec{w}} \circ H}]) \\
 & = d_{H} ([\phi^1_{f_{\vec{v}- \vec{w}} \circ H}], \mathds{1}_M) \\
 & \leq ||{f_{\vec{v} - \vec{w}} \circ H}||_{H}\\
 & = \max({f_{\vec{v}} \circ H}) - \min({f_{\vec{w}} \circ H}) = osc(\vec{v} - \vec{w}). 
 \end{align*}
Note that this computation is true for any $\omega$ without assuming $\omega$ to be aperiodic. 

On the other hand, $f_{\vec{v}} \circ H$ also satisfies the condition that it has no nonconstant contractible periodic orbit. In particular, it has no nonconstant contractible periodic orbit with period at most $1$. By Theorem \ref{sp-computation}, we know 
\[ \rho([M], f_{\vec{v}- \vec{w}} \circ H; \omega) = - \min_{M} (f_{\vec{v}- \vec{w}} \circ H) = \max_{i} (w_i- v_i) \]
and 
\[  \rho([M], f_{\vec{w}- \vec{v}} \circ H; \omega) = - \min_{M} (f_{\vec{w}- \vec{v}} \circ H) = \max_{i} (v_i- w_i).\]
Therefore,  
\[ d_{H} (\Phi(\vec{v}), \Phi(\vec{w})) \geq \max\{\max_{i} (w_i- v_i), \max_{i} (v_i- w_i)\} = |\vec{v} - \vec{w}|_{\infty}\]
where the first inequality comes from (a) in Theorem \ref{Lip}. 

Next, if $\omega$ satisfies the saa-condition, then take a sequence of aperiodic symplectic structures $\omega_n \to \omega$. By Corollary \ref{cor-upper}, for each fixed $\vec{w}$, $\vec{v}$ in $\R^{\infty}$ and any given $\ep>0$, there exists an $N \in \N$ such that whenever $n \geq N$, we have 
\[ \max_{i} (w_i - v_i) = \rho([M], f_{\vec{v}- \vec{w}} \circ H_n; \omega_n) \leq \rho([M], f_{\vec{v}- \vec{w}} \circ H; \omega) + \ep \]
and 
\[ \max_{i} (v_i - w_i) = \rho([M], f_{\vec{w}- \vec{v}} \circ H_n; \omega_n) \leq \rho([M], f_{\vec{w}- \vec{v}} \circ H; \omega) + \ep \]
where the equalities come from the computation above when the symplectic structure is aperiodic. Therefore, 
\[ d_{H} (\Phi(\vec{v}), \Phi(\vec{w})) + \ep \geq \max\{\max_{i} (w_i- v_i), \max_{i} (v_i- w_i)\} = |\vec{v} - \vec{w}|_{\infty}.\]
Since this result is true for any $\ep >0$, we get the conclusion.  \end{proof} 

\begin{remark} If the continuity result of boundary depth is affirmative (especially the lower semicontinuity), a similar argument as in the proof of Theorem \ref{embedding} can imply that if $M$ admits a symplectic structure which satisfies the saa-condition, then there is a quasi-isometric embedding from $(\R^{\infty}, |-|_{\infty})$ into $(\Ham(M, \omega), |-|_{\infty})$. This can be regarded as an ``approximated'' version of Theorem 1.1 in \cite{Ush13}, which can probably cover more symplectic manifolds. \end{remark}

\bibliographystyle{amsplain}
\bibliography{ssp_biblio} 
\noindent \\

\end{document}